\documentclass[11pt]{article}
\usepackage{amsmath}
\usepackage{amsfonts}
\usepackage{amssymb}
\usepackage{mathrsfs}
\usepackage{amsthm}
\usepackage{palatino}
\usepackage[margin=2.5cm, vmargin={1.5cm}]{geometry}

\usepackage{dcolumn}

\usepackage{times}
\usepackage{graphicx}
\usepackage{color}

\usepackage{pstricks}
\usepackage{pst-plot}

\newrgbcolor{LemonChiffon}{1.0 0.98 0.8}
\newrgbcolor{magenta}{1.0 0.3 0.4}
\newrgbcolor{mistyrose}{1.0 0.0 0.6}
\newrgbcolor{deeppink}{1.0 0.08 0.58}
\newrgbcolor{lightsalmon}{1.0 0.63 0.48}
\newrgbcolor{lightred}{0.7 0.0 0.0}
\newrgbcolor{lightblue}{0.68 0.85 0.90}
\newrgbcolor{indianred}{0.80  0.36  0.36}
\newrgbcolor{lightgreen}{0.56 0.93 0.56}
\newrgbcolor{byellow}{0.93 0.76 0.3}

\DeclareMathOperator{\cl}{Cl_2}

\begin{document}

\title{Zeros of the dilogarithm}

\author{Cormac O'Sullivan\footnote{
\newline
Support for this project was provided by a PSC-CUNY Award, jointly funded by The Professional Staff Congress and The City University of New York.}}

%\date{Dec 15, 2014}

\maketitle

%modular symbol
\def\s#1#2{\langle \,#1 , #2 \,\rangle}

%domains
\def\H{{\mathbf{H}}}
\def\F{{\frak F}}
\def\C{{\mathbb C}}
\def\R{{\mathbb R}}
\def\Z{{\mathbb Z}}
\def\Q{{\mathbb Q}}
\def\N{{\mathbb N}}
%symbols
\def\G{{\Gamma}}
\def\GH{{\G \backslash \H}}
\def\g{{\gamma}}
\def\L{{\Lambda}}
\def\ee{{\varepsilon}}
\def\K{{\mathcal K}}
\def\Re{\mathrm{Re}}
\def\Im{\mathrm{Im}}
\def\PSL{\mathrm{PSL}}
\def\SL{\mathrm{SL}}
\def\Vol{\operatorname{Vol}}
\def\lqs{\leqslant}
\def\gqs{\geqslant}
\def\sgn{\operatorname{sgn}}
\def\res{\operatornamewithlimits{Res}}
\def\li{\operatorname{Li_2}}
\def\lip{\operatorname{Li}'_2}
\def\pl{\operatorname{Li}}

\def\clp{\operatorname{Cl}'_2}
\def\clpp{\operatorname{Cl}''_2}

\newcommand{\stira}[2]{{\genfrac{[}{]}{0pt}{}{#1}{#2}}}
\newcommand{\stirb}[2]{{\genfrac{\{}{\}}{0pt}{}{#1}{#2}}}
\newcommand{\norm}[1]{\left\lVert #1 \right\rVert}

%\newcommand{\sect}[1]{{\Large \section{\bf #1}}}
%\newcommand{\pl}[1]{{\operatorname{Li_{#1}}}}

%theorems

\newtheorem{theorem}{Theorem}[section]
\newtheorem{lemma}[theorem]{Lemma}
\newtheorem{prop}[theorem]{Proposition}
\newtheorem{conj}[theorem]{Conjecture}
\newtheorem{cor}[theorem]{Corollary}

\newcounter{coundef}
\newtheorem{adef}[coundef]{Definition}

\renewcommand{\labelenumi}{(\roman{enumi})}

\numberwithin{equation}{section}

\bibliographystyle{alpha}

\begin{abstract}
We show that the dilogarithm has at most one zero on each branch, that each zero is close to a root of unity, and that they may be found to any precision with Newton's method. This work is motivated by applications to the asymptotics of coefficients in partial fraction decompositions considered by Rademacher. We also survey what is known about zeros of polylogarithms in general.
\end{abstract}

\section{Introduction}
In the recent resolution of an old conjecture of Rademacher, described below in Section \ref{rade}, the location of a particular zero, $w_0$, of the dilogarithm played an important role. It has been known since \cite{Ler} that the only zero of the dilogarithm on its principal branch is at $0$. The zero $w_0$ is on the next branch. Zeros on further branches were also needed  in \cite{OS1} and in this paper we locate all  zeros  on every branch.

The dilogarithm is  initially defined as
\begin{equation}\label{def0}
\li(z):=\sum_{n=1}^\infty \frac{z^n}{n^2} \quad \text{ for }|z|\lqs 1,
\end{equation}
see for example \cite{max, Zag07}, with an analytic continuation given by
\begin{equation}\label{li_def}
 -\int_{0}^z \log(1-u) \frac{du}{u}.
\end{equation}
The
principal branch of the logarithm has $-\pi < \arg z \lqs \pi$ with a branch cut $(-\infty,0]$.
From \eqref{li_def}, the
corresponding principal branch of the dilogarithm
has  branch points at $1$, $\infty$ and branch cut $[1,\infty)$. Crossing this branch cut from below, it is easy to show that the dilogarithm increases by $2\pi i \log (z)$ over its principal value. On this new sheet there is now an additional branch point at $0$ coming from the logarithm. In general, the dilogarithm is  a multi-valued holomorphic function with branch points at $1$, $\infty$ and off the principal branch another branch point at $0$. For clarity,
the notations $\li(z)$ and $\log(z)$ will mean the principal branches of these functions from this point. We also note that $\li(z)$ may be expressed as a ${}_3F_2$ hypergeometric function, see \cite[Sect. 2.6]{AAR}, with $\li(z)/z ={}_3F_2(1,1,1;2,2;z)$.

\subsection{Main results}

On any branch it follows, see for example \cite{Ler, max}, that the dilogarithm must take the form
$$
 \phi_{A,B}(z):=\li(z) + 4\pi^2  A +   2\pi i  B  \log \left(z\right) \qquad (A,B \in \Z, \ z \in \C)
$$
and so we want to know when $\phi_{A,B}(z)=0$.

\begin{theorem} \label{ebt}
For $A$, $B \in \Z$, the function $\phi_{A,B}(z)$ has no zeros in $\C$ unless
\begin{enumerate}
\item $B=0$ and $A\gqs 0$ or
\item $-|B|/2<A\lqs |B|/2$.
\end{enumerate}
If (i) or (ii) holds then $\phi_{A,B}(z)$ has exactly one  zero and it is simple.
\end{theorem}

On the principal branch we will see, as noted earlier, that $\li(z)=\phi_{0,0}(z)$ has just  the zero at $z=0$.

\begin{theorem} \label{rhoa2}
For $A\in \Z_{\gqs 1}$, if $\phi_{A,0}(\rho)=0$ then
\begin{equation*}
    \rho =-\exp\left(\pi \sqrt{8A-1/3} \right) +O\left(1/\sqrt{A}\right).
\end{equation*}
With the initial value $-\exp\left(\pi \sqrt{8A-1/3} \right)$, Newton's method applied to $\phi_{A,0}$ produces a sequence converging quadratically to $\rho$.
\end{theorem}

Note that by conjugation,  $\phi_{A,B}(\rho)=0$ if and only if  $\phi_{A,-B}(\overline\rho)=0$. So for $B\neq 0$ we may assume $B \gqs 1$ without loss of generality. The second Bernoulli polynomial is defined as $B_2(x):=x^2-x+1/6$ and   Clausen's integral is
\begin{equation}\label{clau}
\cl(\theta):=-\int_0^\theta \log |2\sin( x/2) | \, dx,
\end{equation}
both shown in Figure \ref{bfig}.

\begin{theorem} \label{new2}
Let $A$ and $B$ be integers satisfying $-B/2<A\lqs B/2$ and suppose $\phi_{A,B}(\rho)=0$.
The sequence $c_0$, $c_1$, $\cdots$
defined as
\begin{equation}\label{roov}
    c_0:= \begin{cases} \exp(2\pi iA/B)  & \text{if \ } A \neq 0 \\ \exp(\pi i/(12B)) & \text{if \ } A=0, \end{cases}
    \qquad \quad
    c_{n+1}:=c_n-\frac{\phi_{A,B}(c_n)}{\phi'_{A,B}(c_n)} \text{ \ \ \ for \ \ $n\gqs 0$},
\end{equation}
 converges quadratically to $\rho$. Using $c_1$, we have
\begin{equation} \label{fdr}
    \rho=e^{2\pi i A/B}\left( 1+\frac{-\cl(2\pi A/B)+i \pi^2 B_2(|A|/B)}{2\pi B}\right) +O\left(\frac{1+\log B}{B^2}\right)
\end{equation}
for an absolute implied constant.
\end{theorem}

The above three theorems are proved in Sections \ref{33} -- \ref{77} after some preliminary results are reviewed in Section \ref{dilogg}.
The dilogarithm is the case $s=2$ of the polylogarithm $\pl_s(z)$.
We put our results in context in Section \ref{polyl} by describing what is known about the zeros of $\pl_s(z)$ for a general fixed $s\in \C$. It turns out that for $\Re(s) \lqs 0$ very much is known, at least for zeros on the principal branch, due to work of Le Roy, Frobenius, Reisz, Peyerimhoff, Sobolev, Gawronski and Stadtm\"uller among others.

\subsection{Rademacher's conjecture} \label{rade}
We briefly describe here the work that led to our study of dilogarithm zeros.
Rademacher conjectured in \cite[p. 302]{Ra} that the  coefficients $C_{hk\ell}(N)$ in the partial fraction decomposition
\begin{equation}\label{tp}
\prod_{j=1}^N \frac{1}{1-q^j}= \frac{C_{011}(N)}{q-1}+ \cdots + \frac{C_{hk\ell}(N)}{(q-e^{2\pi ih/k})^\ell}+ \cdots +\frac{C_{01N}(N)}{(q-1)^N}
\end{equation}
 converge as $N \to \infty$ to the corresponding partial fraction coefficients of the infinite product $\prod_{j=1}^\infty 1/(1-q^j)$. Rademacher had previously given a partial fraction decomposition of this infinite product using his famous exact formula for the partition function $p(n)$, of which it is the generating function, see \cite[pp. 292 - 302]{Ra}. Of course, $\prod_{j=1}^N 1/(1-q^j)$ is the generating function for $p_N(n)$, the number of partitions of $n$ into at most $N$ parts.

Sills and Zeilberger in \cite{SZ} obtained numerical evidence that the conjecture was not correct and in \cite{OS} the true asymptotic behavior of $C_{011}(N)$ was conjectured to be
\begin{equation}\label{conjj1}
   C_{011}(N)=\Re\left[(-2  z_0 e^{-\pi i z_0})\frac{w_0^{-N}}{N^2}\right] +O\left( \frac{|w_0|^{-N}}{N^3}\right)
\end{equation}
for $w_0$ the dilogarithm zero satisfying $\phi_{0,-1}(w_0)=0$ and $z_0$ given by $ w_0=1-e^{2\pi i z_0}$ for $1/2 < \Re(z_0) < 3/2$.  With Theorem \ref{new2} we find
\begin{align*}
    w_0 & \approx 0.91619781620686260140 - 0.18245889720714117505 i, \\
    z_0 & \approx 1.18147496973270876764 + 0.25552764641754743773 i.
\end{align*}
Then $|w_0|<1$, (see the conjugate of $w_0$ in Figure \ref{rfig}), and \eqref{conjj1} implies that $C_{011}(N)$ oscillates with exponentially growing amplitude and diverges. Slightly weaker forms of \eqref{conjj1}, enough to disprove Rademacher's conjecture, were independently shown in \cite{DrGe} and \cite{OS1} with different proofs, but both employing the saddle-point method. The formulation of the main term in \eqref{conjj1} is a little different in \cite{DrGe}.

The proof in \cite{OS1} is based on breaking up $C_{011}(N)$ into manageable components that are similar to Sylvester waves.
 Let $w(A,B)$ denote the zero of $\phi_{A,B}$ when it exists. The appearance of the dilogarithm zero $w_0=w(0,-1)$ in \cite{OS1}, as well as $w(0,-2)$ and $w(1,-3)$, comes from estimating the following sums \eqref{com1}, \eqref{com2} and \eqref{com3} which correspond to some of the largest components of $C_{011}(N)$.  Set
\begin{equation*}
    Q_{hk1}(N):=2\pi i \res_{z=h/k} \frac{  e^{2\pi i  z}}{(1-e^{2\pi i 1 z})(1-e^{2\pi i2 z}) \cdots (1-e^{2\pi i N z})}
\end{equation*}
and
\begin{equation*}
   z_1:=2+\log \bigl(1-w(0,-2)\bigr)/(2\pi i),  \qquad z_3:=3+\log \bigl(1-w(1,-3)\bigr)/(2\pi i).
\end{equation*}
Then it is proved in \cite{OS1,OS2} that
\begin{align}
    2\Re \sum_{N/2<k\lqs N} Q_{1k1}(N) & = \Re\left[2  z_0 e^{-\pi i z_0}\frac{w(0,-1)^{-N}}{N^{2}} \right] + O\left(\frac{|w(0,-1)|^{-N}}{N^{3}}\right), \label{com1}\\
    2\Re \sum_{N/3<k\lqs N/2} Q_{1k1}(N)  & = \Re\left[-\frac{3z_1}{2} e^{-\pi i z_1}\frac{w(0,-2)^{-N}}{N^{2}}  \right] + O\left(\frac{|w(0,-2)|^{-N}}{N^{3}}\right), \label{com2}\\
    2\Re \sum_{N/2<k\lqs N, \ k \text{ odd}} Q_{2k1}(N)  & = \Re\left[-\frac{z_3}{4} e^{-\pi i z_3} \frac{w(1,-3)^{-N}}{N^{2}} \right] + O\left(\frac{|w(1,-3)|^{-N}}{N^{3}}\right). \label{com3}
\end{align}
In fact these are special cases of the results in \cite{OS1,OS2}. See Sections 1, 6 of \cite{OS1} for a more detailed account.
In forthcoming work we also show that $w_0=w(0,-1)$ similarly controls the asymptotics of the first Sylvester waves.

\section{Some properties of the dilogarithm} \label{dilogg}

We have seen that $\phi_{A,B}(z)$ is defined as a single-valued function on $\C$. Away from the cuts $(-\infty,0]$ and $[1,\infty)$ it is holomorphic with
\begin{equation}
    z\phi'_{A,B}(z) = -\log(1-z)+2\pi i B \qquad \text{for} \qquad z \not\in (-\infty,0] \cup [1,\infty).\label{dphi}
\end{equation}
We will also use \eqref{dphi} for $z \in \C-[1,\infty)$ when $B=0$.
For $x\in (-\infty,0)$, as usual,
\begin{equation}\label{wind}
    \lim_{y \to 0^+} \log(x+i y) = \log(x), \qquad \lim_{y \to 0^-} \log(x+i y) = \log(x) -2\pi i.
\end{equation}
Similarly, for $x\in [1,\infty)$
\begin{equation}\label{wind2}
    \lim_{y \to 0^+} \li(x+i y) = \li(x)+2\pi i\log(x), \qquad \lim_{y \to 0^-} \li(x+i y) = \li(x).
\end{equation}
Thus we see that if $z$ makes a full rotation in the positive direction about $0$ then $\phi_{A,B}\to \phi_{A+B,B}$. A  full rotation in the negative direction  about $1$ means
$\phi_{A,B}\to \phi_{A,B+1}$. We have the natural matrix representations
\begin{align}
\begin{pmatrix} A \\ B \\ 1 \end{pmatrix} & \longrightarrow
    \begin{pmatrix}
    1 & 1 & 0 \\
    0 & 1 & 0 \\
    0 & 0 & 1
    \end{pmatrix}
    \begin{pmatrix} A \\ B \\ 1 \end{pmatrix}
     = \begin{pmatrix} A+B \\ B \\ 1 \end{pmatrix}, \label{m0}
    \\
    \begin{pmatrix} A \\ B \\ 1 \end{pmatrix}  & \longrightarrow
    \begin{pmatrix}
    1 & 0 & 0 \\
    0 & 1 & 1 \\
    0 & 0 & 1
    \end{pmatrix}
    \begin{pmatrix} A \\ B \\ 1 \end{pmatrix}
     = \begin{pmatrix} A \\ B+1 \\ 1 \end{pmatrix} \label{m1}
\end{align}
and the $3 \times 3$ matrices in \eqref{m0}, \eqref{m1}  generate the well-known monodromy group of the dilogarithm. This is
 the non-abelian Heisenberg group
\begin{equation*}
    H_3(\Z):=\left\{ \left. \begin{pmatrix}
    1 & x & z \\
    0 & 1 & y \\
    0 & 0 & 1
    \end{pmatrix} \ \right| \ x,y,z \in \Z \right\}
\end{equation*}
 as described in \cite[p. 244]{Ve}, for example.

For all $z \in \C$, except for the given restrictions, the dilogarithm  satisfies the functional equations
\begin{alignat}{2}\label{dilog1}
    \li(1/z) & =-\li(z)-\li(1)-\frac 12 \log^2(-z) \quad \quad & & z \not\in [0,1),\\
    \li(1-z) & =-\li(z)+\li(1)- \log(z)\log(1-z) \quad \quad & & z \neq 0,1. \label{dilog2}
\end{alignat}
Note that $\li(1)=\zeta(2)=\pi^2/6$. Equations \eqref{dilog1} and \eqref{dilog2} are shown in \cite[Sect. 3]{max}, for example, for $z \not\in  [0,\infty)$ and $z \not\in (-\infty,0] \cup [1,\infty)$ respectively. Then use \eqref{wind}, \eqref{wind2} to obtain \eqref{dilog1} and \eqref{dilog2}. For the excluded $z$ values, \eqref{dilog1} becomes
\begin{equation*}
    \li(1/x) -2\pi i \log(x) =-\li(x)-\li(1)-\frac 12 \log^2(-x) \quad \quad  x \in (0,1).
\end{equation*}

Recall $B_2(x)=x^2-x+1/6$ and $\cl(\theta)$ from \eqref{clau}. As in \cite[Sect. 8]{max} we have
\begin{alignat}{2}
    \Re(\li(e^{2\pi i x}) ) & = \sum_{n=1}^\infty \frac{\cos(2\pi n x)}{n^2} = \pi^2 B_2(x-\lfloor x \rfloor) \qquad & & (x\in \R) \label{imc2}\\
    \Im(\li(e^{2\pi i x}) ) & = \sum_{n=1}^\infty \frac{\sin(2\pi n x)}{n^2} = \cl(2\pi x) \qquad & & (x\in \R). \label{imc}
\end{alignat}

%*************************************
% Graphics graph of Clausen

\SpecialCoor
\psset{griddots=5,subgriddiv=0,gridlabels=0pt}
\psset{xunit=0.8cm, yunit=0.5cm}
\psset{linewidth=1pt}
\psset{dotsize=4pt 0,dotstyle=*}

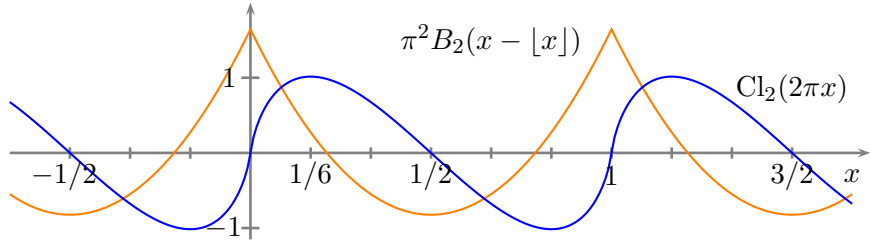
\begin{figure}[h]
\begin{center}
\begin{pspicture}(-5,-2.1)(11,3.3) %\psgrid

\savedata{\mydata}[
{{-4., 1.35326}, {-3.95, 1.29496}, {-3.9, 1.23517}, {-3.85,
  1.17398}, {-3.8, 1.11149}, {-3.75, 1.04778}, {-3.7,
  0.98293}, {-3.65, 0.917028}, {-3.6, 0.850156}, {-3.55,
  0.782392}, {-3.5, 0.713817}, {-3.45, 0.644506}, {-3.4,
  0.574538}, {-3.35, 0.503986}, {-3.3, 0.432927}, {-3.25,
  0.361433}, {-3.2, 0.289578}, {-3.15, 0.217435}, {-3.1,
  0.145077}, {-3.05,
  0.0725742}, {-3., 0}, {-2.95, -0.0725742}, {-2.9,
-0.145077}, {-2.85, -0.217435}, {-2.8, -0.289578}, {-2.75,
-0.361433}, {-2.7, -0.432927}, {-2.65, -0.503986}, {-2.6, -0.574538},
{-2.55, -0.644506}, {-2.5, -0.713817}, {-2.45, -0.782392}, {-2.4,
-0.850156}, {-2.35, -0.917028}, {-2.3, -0.98293}, {-2.25, -1.04778},
{-2.2, -1.11149}, {-2.15, -1.17398}, {-2.1, -1.23517}, {-2.05,
-1.29496}, {-2., -1.35326}, {-1.95, -1.40997}, {-1.9, -1.46501},
{-1.85, -1.51826}, {-1.8, -1.56963}, {-1.75, -1.61901}, {-1.7,
-1.66628}, {-1.65, -1.71134}, {-1.6, -1.75405}, {-1.55, -1.79429},
{-1.5, -1.83193}, {-1.45, -1.86683}, {-1.4, -1.89884}, {-1.35,
-1.9278}, {-1.3, -1.95355}, {-1.25, -1.97592}, {-1.2, -1.99471},
{-1.15, -2.00973}, {-1.1, -2.02075}, {-1.05, -2.02756}, {-1.,
-2.02988}, {-0.95, -2.02746}, {-0.9, -2.01998}, {-0.85, -2.00713},
{-0.8, -1.98852}, {-0.75, -1.96374}, {-0.7, -1.93235}, {-0.65,
-1.8938}, {-0.6, -1.84751}, {-0.55, -1.79277}, {-0.5, -1.72876},
{-0.45, -1.6545}, {-0.4, -1.5688}, {-0.35, -1.47016}, {-0.3,
-1.35668}, {-0.25, -1.22581}, {-0.2, -1.07398}, {-0.15, -0.895776},
{-0.1, -0.682065}, {-0.05, -0.413607}, {0., 0.}, {0.05,
  0.413607}, {0.1, 0.682065}, {0.15, 0.895776}, {0.2, 1.07398}, {0.25,
   1.22581}, {0.3, 1.35668}, {0.35, 1.47016}, {0.4, 1.5688}, {0.45,
  1.6545}, {0.5, 1.72876}, {0.55, 1.79277}, {0.6, 1.84751}, {0.65,
  1.8938}, {0.7, 1.93235}, {0.75, 1.96374}, {0.8, 1.98852}, {0.85,
  2.00713}, {0.9, 2.01998}, {0.95, 2.02746}, {1., 2.02988}, {1.05,
  2.02756}, {1.1, 2.02075}, {1.15, 2.00973}, {1.2, 1.99471}, {1.25,
  1.97592}, {1.3, 1.95355}, {1.35, 1.9278}, {1.4, 1.89884}, {1.45,
  1.86683}, {1.5, 1.83193}, {1.55, 1.79429}, {1.6, 1.75405}, {1.65,
  1.71134}, {1.7, 1.66628}, {1.75, 1.61901}, {1.8, 1.56963}, {1.85,
  1.51826}, {1.9, 1.46501}, {1.95, 1.40997}, {2., 1.35326}, {2.05,
  1.29496}, {2.1, 1.23517}, {2.15, 1.17398}, {2.2, 1.11149}, {2.25,
  1.04778}, {2.3, 0.98293}, {2.35, 0.917028}, {2.4, 0.850156}, {2.45,
  0.782392}, {2.5, 0.713817}, {2.55, 0.644506}, {2.6,
  0.574538}, {2.65, 0.503986}, {2.7, 0.432927}, {2.75,
  0.361433}, {2.8, 0.289578}, {2.85, 0.217435}, {2.9,
  0.145077}, {2.95, 0.0725742}, {3.,
  0}, {3.05, -0.0725742}, {3.1, -0.145077}, {3.15,
-0.217435}, {3.2, -0.289578}, {3.25, -0.361433}, {3.3, -0.432927},
{3.35, -0.503986}, {3.4, -0.574538}, {3.45, -0.644506}, {3.5,
-0.713817}, {3.55, -0.782392}, {3.6, -0.850156}, {3.65, -0.917028},
{3.7, -0.98293}, {3.75, -1.04778}, {3.8, -1.11149}, {3.85, -1.17398},
{3.9, -1.23517}, {3.95, -1.29496}, {4., -1.35326}, {4.05, -1.40997},
{4.1, -1.46501}, {4.15, -1.51826}, {4.2, -1.56963}, {4.25, -1.61901},
{4.3, -1.66628}, {4.35, -1.71134}, {4.4, -1.75405}, {4.45, -1.79429},
{4.5, -1.83193}, {4.55, -1.86683}, {4.6, -1.89884}, {4.65, -1.9278},
{4.7, -1.95355}, {4.75, -1.97592}, {4.8, -1.99471}, {4.85, -2.00973},
{4.9, -2.02075}, {4.95, -2.02756}, {5., -2.02988}, {5.05, -2.02746},
{5.1, -2.01998}, {5.15, -2.00713}, {5.2, -1.98852}, {5.25, -1.96374},
{5.3, -1.93235}, {5.35, -1.8938}, {5.4, -1.84751}, {5.45, -1.79277},
{5.5, -1.72876}, {5.55, -1.6545}, {5.6, -1.5688}, {5.65, -1.47016},
{5.7, -1.35668}, {5.75, -1.22581}, {5.8, -1.07398}, {5.85,
-0.895776}, {5.9, -0.682065}, {5.95, -0.413607}, {6.,
0}, {6.05, 0.413607}, {6.1, 0.682065}, {6.15,
  0.895776}, {6.2, 1.07398}, {6.25, 1.22581}, {6.3, 1.35668}, {6.35,
  1.47016}, {6.4, 1.5688}, {6.45, 1.6545}, {6.5, 1.72876}, {6.55,
  1.79277}, {6.6, 1.84751}, {6.65, 1.8938}, {6.7, 1.93235}, {6.75,
  1.96374}, {6.8, 1.98852}, {6.85, 2.00713}, {6.9, 2.01998}, {6.95,
  2.02746}, {7., 2.02988}, {7.05, 2.02756}, {7.1, 2.02075}, {7.15,
  2.00973}, {7.2, 1.99471}, {7.25, 1.97592}, {7.3, 1.95355}, {7.35,
  1.9278}, {7.4, 1.89884}, {7.45, 1.86683}, {7.5, 1.83193}, {7.55,
  1.79429}, {7.6, 1.75405}, {7.65, 1.71134}, {7.7, 1.66628}, {7.75,
  1.61901}, {7.8, 1.56963}, {7.85, 1.51826}, {7.9, 1.46501}, {7.95,
  1.40997}, {8., 1.35326}, {8.05, 1.29496}, {8.1, 1.23517}, {8.15,
  1.17398}, {8.2, 1.11149}, {8.25, 1.04778}, {8.3, 0.98293}, {8.35,
  0.917028}, {8.4, 0.850156}, {8.45, 0.782392}, {8.5,
  0.713817}, {8.55, 0.644506}, {8.6, 0.574538}, {8.65,
  0.503986}, {8.7, 0.432927}, {8.75, 0.361433}, {8.8,
  0.289578}, {8.85, 0.217435}, {8.9, 0.145077}, {8.95,
  0.0725742}, {9.,
  0}, {9.05, -0.0725742}, {9.1, -0.145077}, {9.15,
-0.217435}, {9.2, -0.289578}, {9.25, -0.361433}, {9.3, -0.432927},
{9.35, -0.503986}, {9.4, -0.574538}, {9.45, -0.644506}, {9.5,
-0.713817}, {9.55, -0.782392}, {9.6, -0.850156}, {9.65, -0.917028},
{9.7, -0.98293}, {9.75, -1.04778}, {9.8, -1.11149}, {9.85, -1.17398},
{9.9, -1.23517}, {9.95, -1.29496}, {10., -1.35326}}
]

\savedata{\mydatb}[
{{-4., -1.09662}, {-3.95, -1.15008}, {-3.9, -1.2008}, {-3.85,
-1.24878}, {-3.8, -1.29401}, {-3.75, -1.33651}, {-3.7, -1.37626},
{-3.65, -1.41327}, {-3.6, -1.44754}, {-3.55, -1.47907}, {-3.5,
-1.50786}, {-3.45, -1.5339}, {-3.4, -1.5572}, {-3.35, -1.57777},
{-3.3, -1.59559}, {-3.25, -1.61066}, {-3.2, -1.623}, {-3.15,
-1.6326}, {-3.1, -1.63945}, {-3.05, -1.64356}, {-3., -1.64493},
{-2.95, -1.64356}, {-2.9, -1.63945}, {-2.85, -1.6326}, {-2.8,
-1.623}, {-2.75, -1.61066}, {-2.7, -1.59559}, {-2.65, -1.57777},
{-2.6, -1.5572}, {-2.55, -1.5339}, {-2.5, -1.50786}, {-2.45,
-1.47907}, {-2.4, -1.44754}, {-2.35, -1.41327}, {-2.3, -1.37626},
{-2.25, -1.33651}, {-2.2, -1.29401}, {-2.15, -1.24878}, {-2.1,
-1.2008}, {-2.05, -1.15008}, {-2., -1.09662}, {-1.95, -1.04042},
{-1.9, -0.981477}, {-1.85, -0.919792}, {-1.8, -0.855366}, {-1.75,
-0.788198}, {-1.7, -0.718288}, {-1.65, -0.645637}, {-1.6, -0.570244},
{-1.55, -0.492109}, {-1.5, -0.411234}, {-1.45, -0.327616}, {-1.4,
-0.241257}, {-1.35, -0.152156}, {-1.3, -0.0603142}, {-1.25,
  0.0342695}, {-1.2, 0.131595}, {-1.15, 0.231662}, {-1.1,
  0.33447}, {-1.05, 0.44002}, {-1., 0.548311}, {-0.95,
  0.659344}, {-0.9, 0.773119}, {-0.85, 0.889635}, {-0.8,
  1.00889}, {-0.75, 1.13089}, {-0.7, 1.25563}, {-0.65,
  1.38312}, {-0.6, 1.51334}, {-0.55, 1.6463}, {-0.5, 1.78201}, {-0.45,
   1.92046}, {-0.4, 2.06165}, {-0.35, 2.20558}, {-0.3,
  2.35226}, {-0.25, 2.50167}, {-0.2, 2.65383}, {-0.15,
  2.80872}, {-0.1, 2.96636}, {-0.05, 3.12675}, {0., 3.28987}, {0.05,
  3.12675}, {0.1, 2.96636}, {0.15, 2.80872}, {0.2, 2.65383}, {0.25,
  2.50167}, {0.3, 2.35226}, {0.35, 2.20558}, {0.4, 2.06165}, {0.45,
  1.92046}, {0.5, 1.78201}, {0.55, 1.6463}, {0.6, 1.51334}, {0.65,
  1.38312}, {0.7, 1.25563}, {0.75, 1.13089}, {0.8, 1.00889}, {0.85,
  0.889635}, {0.9, 0.773119}, {0.95, 0.659344}, {1., 0.548311}, {1.05,
   0.44002}, {1.1, 0.33447}, {1.15, 0.231662}, {1.2, 0.131595}, {1.25,
   0.0342695}, {1.3, -0.0603142}, {1.35, -0.152156}, {1.4,
-0.241257}, {1.45, -0.327616}, {1.5, -0.411234}, {1.55, -0.492109},
{1.6, -0.570244}, {1.65, -0.645637}, {1.7, -0.718288}, {1.75,
-0.788198}, {1.8, -0.855366}, {1.85, -0.919792}, {1.9, -0.981477},
{1.95, -1.04042}, {2., -1.09662}, {2.05, -1.15008}, {2.1, -1.2008},
{2.15, -1.24878}, {2.2, -1.29401}, {2.25, -1.33651}, {2.3, -1.37626},
{2.35, -1.41327}, {2.4, -1.44754}, {2.45, -1.47907}, {2.5, -1.50786},
{2.55, -1.5339}, {2.6, -1.5572}, {2.65, -1.57777}, {2.7, -1.59559},
{2.75, -1.61066}, {2.8, -1.623}, {2.85, -1.6326}, {2.9, -1.63945},
{2.95, -1.64356}, {3., -1.64493}, {3.05, -1.64356}, {3.1, -1.63945},
{3.15, -1.6326}, {3.2, -1.623}, {3.25, -1.61066}, {3.3, -1.59559},
{3.35, -1.57777}, {3.4, -1.5572}, {3.45, -1.5339}, {3.5, -1.50786},
{3.55, -1.47907}, {3.6, -1.44754}, {3.65, -1.41327}, {3.7, -1.37626},
{3.75, -1.33651}, {3.8, -1.29401}, {3.85, -1.24878}, {3.9, -1.2008},
{3.95, -1.15008}, {4., -1.09662}, {4.05, -1.04042}, {4.1, -0.981477},
{4.15, -0.919792}, {4.2, -0.855366}, {4.25, -0.788198}, {4.3,
-0.718288}, {4.35, -0.645637}, {4.4, -0.570244}, {4.45, -0.492109},
{4.5, -0.411234}, {4.55, -0.327616}, {4.6, -0.241257}, {4.65,
-0.152156}, {4.7, -0.0603142}, {4.75, 0.0342695}, {4.8,
  0.131595}, {4.85, 0.231662}, {4.9, 0.33447}, {4.95, 0.44002}, {5.,
  0.548311}, {5.05, 0.659344}, {5.1, 0.773119}, {5.15,
  0.889635}, {5.2, 1.00889}, {5.25, 1.13089}, {5.3, 1.25563}, {5.35,
  1.38312}, {5.4, 1.51334}, {5.45, 1.6463}, {5.5, 1.78201}, {5.55,
  1.92046}, {5.6, 2.06165}, {5.65, 2.20558}, {5.7, 2.35226}, {5.75,
  2.50167}, {5.8, 2.65383}, {5.85, 2.80872}, {5.9, 2.96636}, {5.95,
  3.12675}, {6., 3.28987}, {6.05, 3.12675}, {6.1, 2.96636}, {6.15,
  2.80872}, {6.2, 2.65383}, {6.25, 2.50167}, {6.3, 2.35226}, {6.35,
  2.20558}, {6.4, 2.06165}, {6.45, 1.92046}, {6.5, 1.78201}, {6.55,
  1.6463}, {6.6, 1.51334}, {6.65, 1.38312}, {6.7, 1.25563}, {6.75,
  1.13089}, {6.8, 1.00889}, {6.85, 0.889635}, {6.9, 0.773119}, {6.95,
  0.659344}, {7., 0.548311}, {7.05, 0.44002}, {7.1, 0.33447}, {7.15,
  0.231662}, {7.2, 0.131595}, {7.25,
  0.0342695}, {7.3, -0.0603142}, {7.35, -0.152156}, {7.4, -0.241257},
{7.45, -0.327616}, {7.5, -0.411234}, {7.55, -0.492109}, {7.6,
-0.570244}, {7.65, -0.645637}, {7.7, -0.718288}, {7.75, -0.788198},
{7.8, -0.855366}, {7.85, -0.919792}, {7.9, -0.981477}, {7.95,
-1.04042}, {8., -1.09662}, {8.05, -1.15008}, {8.1, -1.2008}, {8.15,
-1.24878}, {8.2, -1.29401}, {8.25, -1.33651}, {8.3, -1.37626}, {8.35,
-1.41327}, {8.4, -1.44754}, {8.45, -1.47907}, {8.5, -1.50786}, {8.55,
-1.5339}, {8.6, -1.5572}, {8.65, -1.57777}, {8.7, -1.59559}, {8.75,
-1.61066}, {8.8, -1.623}, {8.85, -1.6326}, {8.9, -1.63945}, {8.95,
-1.64356}, {9., -1.64493}, {9.05, -1.64356}, {9.1, -1.63945}, {9.15,
-1.6326}, {9.2, -1.623}, {9.25, -1.61066}, {9.3, -1.59559}, {9.35,
-1.57777}, {9.4, -1.5572}, {9.45, -1.5339}, {9.5, -1.50786}, {9.55,
-1.47907}, {9.6, -1.44754}, {9.65, -1.41327}, {9.7, -1.37626}, {9.75,
-1.33651}, {9.8, -1.29401}, {9.85, -1.24878}, {9.9, -1.2008}, {9.95,
-1.15008}, {10., -1.09662}}
]

\psline[linecolor=gray]{->}(-4,0)(10.3,0)
\psline[linecolor=gray]{->}(0,-2.3)(0,4)
\psline[linecolor=gray](-0.15,2)(0.15,2)
\psline[linecolor=gray](-0.15,-2)(0.15,-2)
\multirput(-3,-0.15)(1,0){13}{\psline[linecolor=gray](0,0)(0,0.3)}

\dataplot[linecolor=orange,linewidth=0.8pt,plotstyle=line]{\mydatb}
\dataplot[linecolor=blue,linewidth=0.8pt,plotstyle=line]{\mydata}

\rput(3,-0.6){$1/2$}
\rput(-3.1,-0.6){$-1/2$}
\rput(6,-0.6){$1$}
\rput(9,-0.6){$3/2$}
\rput(10,-0.6){$x$}
\rput(1,-0.6){$1/6$}
%\rput(2,2.3){$(\pi/3,\mathcal K)$}
\rput(-0.35,2){$1$}
\rput(-0.45,-2){$-1$}
\rput(9,1.7){$\cl(2\pi x)$}
\rput(4,3){$\pi^2 B_2(x-\lfloor x \rfloor)$}

%\psdots(1,2.02988)

\end{pspicture}
\caption{Real and imaginary parts of $\li(e^{2\pi i x})$}\label{bfig}
\end{center}
\end{figure}
%*************************************

\noindent
The function $\cl(\theta)$ is odd, has period $2\pi$ and satisfies $2\cl(\theta) - 2\cl(\pi - \theta) = \cl(2\theta)$.
Then
\begin{equation}\label{byo}
\clp(\theta)=- \log |2\sin( \theta/2) |, \qquad \clpp(\theta)=-\frac 12 \cot(\theta/2)
\end{equation}
and  $|\cl(\theta)|$ has its maximum at $\theta=\pi/3$, for example, with maximum value
\begin{equation}\label{kap}
\kappa := \sum_{n=1}^\infty \frac{\sin(\pi n/3)}{n^2} \approx 1.0149416.
\end{equation}

From \cite[(11.1)]{Ra},
\begin{equation}\label{cotr}
\cot(z) = \frac 1z - \sum_{n=1}^\infty \frac{2^{2n}|B_{2n}|}{(2n)!} z^{2n-1} \qquad (0<|z|<\pi)
\end{equation}
for $B_{2n}$  the Bernoulli number.
Integrating \eqref{cotr} twice we find
\begin{equation}\label{serc}
\cl(\theta) = \theta - \theta \log |\theta| + \sum_{n=1}^\infty \frac{|B_{2n}|}{2n (2n+1)!} \theta^{2n+1} \qquad (-2\pi< \theta < 2\pi)
\end{equation}
and as in \cite[Lemma 3.7]{OS1} it follows from \eqref{serc} that
\begin{equation}\label{dtb}
|\cl(\theta)| \lqs |\theta|  - |\theta|  \log |\theta|  +  |\theta| ^{3}/54 \qquad (-\pi < \theta < \pi).
\end{equation}

\section{Zeros of $\phi_{A,0}$} \label{33}
In this section we locate the zeros of
$\phi_{A,0}(z)=\li(z)+4\pi^2 A$.

\begin{lemma} \label{im0}
We have $\Im \li(z)=0$ if and only if $z \in (-\infty, 1]$.
\end{lemma}
\begin{proof}
We first note that $\overline{\li(\overline{z})} = \li(z)$
is true for $|z|\lqs 1$ by \eqref{def0}, so it must be true for all $z$ in $\C-[1,\infty)$ by analytic continuation. It follows that
\begin{equation*}
    2i \Im \li(z) = \li(z) - \overline{\li(z)} = \li(z) - \li(\overline{z}) \qquad (z \in \C-[1,\infty))
\end{equation*}
and hence
\begin{equation} \label{r1}
    \Im \li(z)=0 \quad \text{ for } z \in (-\infty, 1].
\end{equation}
With \eqref{r1} and \eqref{dilog1} we can see that
\begin{equation} \label{r2}
    \Im \li(z)=-\pi  \log z \neq 0 \quad \text{ for } z \in (1,\infty).
\end{equation}

Write $z \in \C$ as $z=r e^{i \theta}$ for $r\gqs 0$ and $-\pi<\theta\lqs \pi$.
From \eqref{li_def} we have that
\begin{equation*}
    \frac d{dr} \li( r e^{i \theta}) = - \frac 1r \log(1-r e^{i \theta})
\end{equation*}
and so
\begin{equation} \label{plrd}
    \frac d{dr} \Im \li( r e^{i \theta}) = - \Im\left( \frac 1r \log(1-r e^{i \theta})\right) = - \frac 1r \arg(1-r e^{i \theta}).
\end{equation}
 Hence $\Im \li( r e^{i \theta})$ is a  strictly increasing function of $r$ for $0<\theta<\pi$ and  strictly decreasing for $-\pi<\theta<0$. Since $\Im \li( 0) =0$ it follows that $\Im \li(z) \neq 0$ for $z\in \C-\R$. Combining this with \eqref{r1}, \eqref{r2} completes the proof.
\end{proof}

\begin{prop}
Let $A\in \Z$. Then $\phi_{A,0}(z)$ has one zero
 if  $A \gqs 0$ and no zeros otherwise.
\end{prop}
\begin{proof}
With Lemma \ref{im0}, we see the only possible solutions to $\phi_{A,0}(z)=0$  have $z=x \in (-\infty,1]$.
We know $\li(x)$ is continuous on $(-\infty,1]$ and it is real-valued for these $x$ by \eqref{r1}. We have
\begin{equation} \label{ddx}
    \frac{d}{dx} \li(x)=-\frac 1x \log(1-x) >0 \quad \text{ for } \quad x \in (-\infty,1)
\end{equation}
so that $\li(x)$ is  strictly increasing. If $\phi_{A,0}(x)=0$  then
\begin{equation*}
       \pi^2/6 = \li(1) \gqs  \li(x) = -4\pi^2  A
\end{equation*}
implying that $\phi_{A,0}$ has no zeros for $A<0$. For $A=0$ there is the necessarily unique zero at $x=0$.

Now we assume $A\gqs 1$.
Then
\begin{equation}\label{ebo}
    \phi_{A,0}(-x)=0 \implies -x \in (-\infty,-1) \quad \text{ since } \quad \li(-1)=-\pi^2/12.
\end{equation}
 From the functional equation \eqref{dilog1} we see
\begin{equation}\label{eqq1}
    \li(-x) = -\frac{\pi^2}{6}-\frac 12 \log^2(x) -\li\left( -\frac 1x\right) \qquad (x>0).
\end{equation}
Note that
\begin{equation}\label{eqq2}
    \left|\li(z)\right| \lqs \sum_{n=1}^\infty \frac{|z|^n}{n^2} \lqs |z|\sum_{n=1}^\infty \frac{1}{n^2} = |z|\frac{\pi^2}6 \qquad (|z|\lqs 1)
\end{equation}
and so \eqref{eqq1} and \eqref{eqq2} imply that, for $x \gqs 1$,
\begin{equation}\label{eqq3}
    \li(-x) = -\frac{\pi^2}{6}-\frac 12 \log^2(x) + \varepsilon \qquad
\end{equation}
with $\varepsilon \in \R$ satisfying $|\varepsilon| \lqs \pi^2/(6x)$. Therefore $\li(-x) \to -\infty$ as $-x \to -\infty$ and $\li(-x) = -4\pi^2  A$ has a single solution, as required.
\end{proof}

\begin{prop} \label{rhoa}
For $A\in \Z_{\gqs 1}$, if $\phi_{A,0}(\rho)=0$ then
\begin{equation*}
    \left|\rho +\exp\left(\pi \sqrt{8A-1/3} \right)\right| < 1/\sqrt{A}.
\end{equation*}
\end{prop}
\begin{proof}
It follows from \eqref{eqq3} that
\begin{equation}\label{eqq4}
    \rho = -\exp\left( \pi \sqrt{8A-1/3 + \varepsilon'}\right) \quad \text{ for } \quad |\varepsilon'|=\frac{2|\varepsilon|}{\pi^2} \lqs \frac{1}{3|\rho|}.
\end{equation}
We have $-1/3 \lqs \varepsilon' \lqs 1/3$ by \eqref{ebo} and therefore
\begin{equation*}
    \rho \lqs -\exp\left( \pi \sqrt{8A-2/3}\right)
\end{equation*}
which in turn implies
\begin{equation*}
    |\varepsilon'| \lqs \exp\left( -\pi \sqrt{8A-2/3}\right)/3.
\end{equation*}
Next write
\begin{equation*}
     \exp\left( \pi \sqrt{8A-1/3 + \varepsilon'}\right) =  \exp\left( \pi \sqrt{8A-1/3 }\left(1+u \right)^{1/2}\right)
\end{equation*}
with $u=\varepsilon'/(8A-1/3)$.  If $u \in \C$ satisfies $|u|\lqs 1/2$, say, we have the simple bounds
\begin{align}\label{cb1}
    (1+u)^{1/2}& =1+w \quad \implies \quad |w| \lqs |u|, \\
    \exp(u) & =1+w \quad \implies \quad |w| \lqs 2|u| \label{cb2}.
\end{align}
 Hence
\begin{align*}
    \exp\left( \pi \sqrt{8A-1/3 }\left(1+u \right)^{1/2}\right) & = \exp\left( \pi \sqrt{8A-1/3 }\left(1+w \right)\right)\\
    & = \exp\left( \pi \sqrt{8A-1/3 }\right)\exp\left( w  \cdot \pi \sqrt{8A-1/3 }\right)\\
     & = \exp\left( \pi \sqrt{8A-1/3 }\right)\left( 1+w' \right)
\end{align*}
for $|w'| \lqs 2|w  \cdot \pi \sqrt{8A-1/3 }| \lqs 2|u \cdot \pi \sqrt{8A-1/3 }|$ using \eqref{cb1} and \eqref{cb2}. Then the error has the bounds
\begin{align*}
    \left|w'\exp\left( \pi \sqrt{8A-1/3 }\right)\right| & \lqs \frac{2\pi}{3\sqrt{8A-1/3}}\exp\left( \pi \sqrt{8A-1/3 }- \pi \sqrt{8A-2/3}\right)\\
    & \lqs \frac{2\pi}{3\sqrt{8A-1/3}}\exp\left( \pi \sqrt{8-1/3 }- \pi \sqrt{8-2/3}\right) < \frac{1}{\sqrt{A}}
\end{align*}
as required, where we used that $\sqrt{x+c}-\sqrt{x}$ is a decreasing function of $x$ for $x+c,x\gqs 0$.
\end{proof}

\section{Newton's method for $\phi_{A,0}$} \label{44}

\begin{prop} \label{d0}
Let $z$ and $c$ be two points in $\C$ with $\Re(z)$, $\Re(c) \lqs 1-s$ for $s>0$.
Then
\begin{equation*}
    \left| \li (z) - \li (c) - (z-c) \lip (c)\right| \lqs |z-c|^2 \frac{1}{2s}.
\end{equation*}
\end{prop}
\begin{proof}
We need to bound the remainder term in the following Taylor expansion of $\li(z)$ at $c$,
\begin{equation*}
    \li (z) = \li (c) + (z-c) \lip (c) + \frac{(z-c)^2}{2\pi i}\int_C \frac{\li(w)}{(w-c)^2(w-z)} \, dw,
\end{equation*}
where  $C$ is the circular path of radius $T$ centered at $1$ that avoids the branch cut $[1,\infty)$ as shown in Figure \ref{cir}.
%*************************************
% Graphics graph of C
\SpecialCoor
\psset{griddots=5,subgriddiv=0,gridlabels=0pt}
\psset{xunit=0.7cm, yunit=0.7cm, runit=0.7cm}
\psset{linewidth=1pt}
\psset{dotsize=4pt 0,dotstyle=*}
\begin{figure}[h]
\begin{center}
\begin{pspicture}(-2,-2.5)(4,3) %\psgrid

\psline[linecolor=gray]{}(-3,0)(1,0)
\psline[linecolor=gray]{->}(0,-2.5)(0,3)
\multirput(-2,-0.08)(1,0){2}{\psline[linecolor=gray](0,0)(0,0.16)}
\multirput(-0.1,-2)(0,1){5}{\psline[linecolor=gray](0,0)(0.2,0)}

\multirput(1,-0.08)(1,0){4}{\psline[linecolor=red](0,0)(0,0.16)}
%\psline[linecolor=red]{}(1,-0.15)(1,0.15)
%\psline[linecolor=red]{}(4,-0.15)(4,0.15)
\psline[linecolor=red]{->}(1,0)(5,0)

%\pscircle[linecolor=blue](1,0){2.5}
%\pscircle[linecolor=blue](1,0){0.5}
%\rput(3,-0.6){$1/2$}
\newrgbcolor{darkbrn}{0.4 0.2 0}

\psset{arrowscale=2,arrowinset=0.5}
\psarc[linecolor=darkbrn](1,0){0.5}{23.58}{336.42}
\psarc[linecolor=darkbrn](1,0){2.5}{4.59}{355.41}
%\psline[linecolor=darkbrn](1.458,0.2)(3.5,0.2)
\psline[linecolor=darkbrn](1.44,0.2)(3.51,0.2)
\psline[linecolor=darkbrn]{->}(2.9,0.2)(3,0.2)
\psline[linecolor=darkbrn](1.44,-0.2)(3.51,-0.2)
\psline[linecolor=darkbrn]{->}(1.9,-0.2)(1.8,-0.2)
\psarc[linecolor=darkbrn]{->}(1,0){2.5}{60}{62}

\psdots(-0.3,0.6)
\psdots(-0.5,-1.4)
\rput(-0.5,0.9){$z$}
\rput(-0.6,-1.0){$c$}
\rput(3.3,2){$C$}

\end{pspicture}
\caption{The path of integration $C$}\label{cir}
\end{center}
\end{figure}
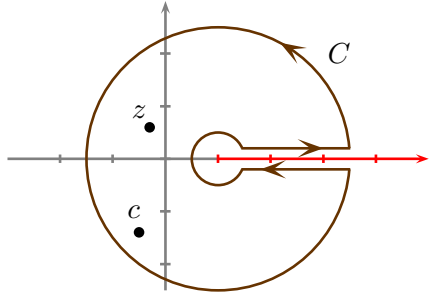
%*************************************
The path $C$ makes a small circle of radius $\varepsilon$ about $1$ that is connected to a path  from $1+\varepsilon$ to $1+T$ just above the cut and a path  from $1+T$  to $1+\varepsilon$ just below the cut. Letting  the horizontal paths meet the branch cut, we find the values of $\li(w)$ with \eqref{wind2}. Use \eqref{dilog1} and \eqref{dilog2} to see that the growth of $\li(z)$ is logarithmic on the circles of radius $T$ and $\varepsilon$ so that
as $T \to \infty$ and $\varepsilon \to 0$  their contributions to $\int_C$ go to zero.  Therefore
\begin{equation*}
    \frac{1}{2\pi i}\int_C \frac{\li(w)}{(w-c)^2(w-z)} \, dw = \int_{0}^\infty \frac{\log( t +1)}{( t+1 -c)^2( t+1 -z)} \, dt.
\end{equation*}
Since $| t+1 -c|$, $| t+1 -z|  \gqs t+s$ and $|\log(t+1)| \lqs  t$, the proposition follows from
\begin{equation*}
    \left|\int_{0}^\infty \frac{ \log( t +1)}{(t+1 -c)^2(t+1 -z)} \, dt \right| \lqs \int_{0}^\infty \frac{t}{(t+s)^3} \, dt = \frac{1}{2s}. \qedhere
\end{equation*}
\end{proof}

Write $D_A:=\exp \left(\pi\sqrt{8A-1/3}\right)$ and $c^*:=c-\phi_{A,0}(c)/\phi'_{A,0}(c)$.

\begin{lemma} \label{nwt0}
Let $A \in \Z_{\gqs 1}$. Suppose $\rho$ and $c$ are in the interval $\Bigl[-D_A-1/\sqrt{A}, -D_A+1/\sqrt{A}\Bigr]$  and $\phi_{A,0}(\rho)=0$. Then
\begin{equation} \label{mon0}
    \left| \rho - c^*\right| < \frac{|\rho-c|^2}{2\pi \sqrt{A}}.
\end{equation}
\end{lemma}
\begin{proof}
Let $z=\rho$ and $s=D_A$ in Proposition \ref{d0} to get
 \begin{equation} \label{sw7}
    \bigl| \phi_{A,0}(\rho) - \phi_{A,0}(c) - (\rho-c) \phi'_{A,0}(c)\bigr| \lqs |\rho-c|^2 \frac{1}{2D_A}.
\end{equation}
Now $c \cdot \phi'_{A,0}(c) = -\log(1-c)$ as in \eqref{dphi} and dividing \eqref{sw7} by $\phi'_{A,0}(c)$ finds
\begin{equation*}
    \left| \rho - c^*\right| \lqs |\rho-c|^2 \frac{|c|}{2D_A |\log(1-c)|} \lqs |\rho-c|^2 \frac{D_A+1/\sqrt{A}}{2D_A \cdot \pi\sqrt{8A-1/3}} < \frac{|\rho-c|^2}{2\pi \sqrt{A}}.
    \qedhere
\end{equation*}
 \end{proof}

Let $c_0:=-D_A = -\exp(\pi\sqrt{8A-1/3})$ and  $c_{n+1}:=c_n-\phi_{A,0}(c)/\phi'_{A,0}(c)$ for $n\gqs 0$.
\begin{prop} \label{seq}
Suppose $A \in \Z_{\gqs 1}$ and $\rho$ satisfies $\phi_{A,0}(\rho)=0$. Then the sequence $c_0$, $c_1$, $\cdots$ above converges  to $\rho$ with
\begin{equation*}
    \left| \rho - c_n\right| \lqs  2\pi \sqrt{A} \cdot (2\pi A)^{-2^n}.
\end{equation*}
\end{prop}
\begin{proof}
With Proposition \ref{rhoa} we know that $\rho$ is in the interval $[-D_A-1/\sqrt{A}, -D_A+1/\sqrt{A}]$.
Lemma \ref{nwt0} implies that for any $c \in [-D_A-1/\sqrt{A}, -D_A+1/\sqrt{A}]$ we have $|\rho-c^*|<|\rho-c|$ so that $c^*$ remains in the interval. In this way, starting with $c_0$, we may keep applying Lemma \ref{nwt0} to each term in the sequence.
We find
\begin{equation*}
    \frac{|\rho-c_n|}{2\pi \sqrt{A}} < \left(\frac{|\rho-c_{n-1}|}{2\pi \sqrt{A}}\right)^2 < \left(\frac{|\rho-c_{n-2}|}{2\pi \sqrt{A}}\right)^4 <
    \cdots < \left(\frac{|\rho-c_{0}|}{2\pi \sqrt{A}}\right)^{2^n}
\end{equation*}
and $|\rho-c_0| \lqs 1/\sqrt{A}$ completes the proof.
\end{proof}

For example, with $A=1$ the sequence produced by Proposition \ref{seq} has initial value $c_0 \approx -5994.97063$ and next term  $c_1 \approx -5995.08558$. Taking more terms, we find $\phi_{1,0}(\rho)=0$ for $\rho \approx -5995.08558$  which is correct to the precision shown and already given by $c_1$.

With the results in Sections \ref{33} and \ref{44} we have proved Theorem \ref{rhoa2} and the $B=0$ case of Theorem \ref{ebt}. We prove the simplicity of all zeros in Proposition \ref{simple}.

\section{Zeros of $\phi_{A,B}$} \label{55}

\begin{lemma}
For $A\in \Z$, $B \in \Z_{\gqs 1}$ and $x\in \R$ we have $\phi_{A,B}(x)\neq 0$.
\end{lemma}
\begin{proof}
For $x > 1$, the imaginary part of $\phi_{A,B}(x)$ equals
$\pi (-1+2B)\log x \neq 0$ by \eqref{r2}.
For $0 \lqs x \lqs 1$ the imaginary part of $\phi_{A,B}(x)$ equals $2\pi B \log x$ by \eqref{r1}.
This is zero only if $x=1$, but then
\begin{equation*}
    \phi_{A,B}(1) = \li(1) + 4\pi^2  A +   2\pi i  B  \log \left(1\right) = \pi^2(1/6 +4A) \neq 0.
\end{equation*}
 For the remaining case of $x<0$ we have $\log x=\log |x| +\pi i$ so that the imaginary part of $\phi_{A,B}(x)$ equals $2\pi B \log |x|$. This is zero only if $x=-1$, but then
\begin{equation*}
    \phi_{A,B}(-1) = \li(-1) + 4\pi^2  A - 2\pi^2  B   = \pi^2(-1/12 + 4A-2B ) \neq 0. \qedhere
\end{equation*}
\end{proof}

So we look for solutions to $\phi_{A,B}(z)=0$ with $z \in \C -\R$. The main idea to locate these zeros is to consider the vanishing of the real and imaginary parts of $\phi_{A,B}(z)$ separately. We will see that $\Im \, \phi_{A,B}(z)=0$ makes a curve near the unit circle and $\Re \, \phi_{A,B}(z)=0$
makes a curve close to the ray from the origin through $e^{2\pi i A/B}$ when $A$ is small enough.

For $B \in \Z_{\gqs 1}$  fixed, write the imaginary part of
\begin{equation} \label{mainp}
   \phi_{A,B}(r e^{i \theta}) =  \li(r e^{i \theta}) + 4\pi^2  A +   2\pi i  B  \log \left(r e^{i \theta}\right).
\end{equation}
 as the function
\begin{equation} \label{plr}
    I_\theta(r):=\Im \li(r e^{i \theta}) +   2\pi   B  \log r  .
\end{equation}
As in \eqref{plrd},
\begin{align} \label{plrd2}
 \frac{\partial I}{\partial r} & = \frac 1r \left(2\pi B -  \arg(1-r e^{i \theta})\right),\\
 \frac{\partial  I}{\partial \theta} & =  -  \log |1-r e^{i \theta}|. \label{plrd3}
\end{align}

\begin{lemma}
Fix $B \in \Z_{\gqs 1}$. For each $\theta$ with $0<|\theta|<\pi$ there exists a unique $r>0$ so that $I_\theta(r)=0$.
\end{lemma}
\begin{proof}
With \eqref{plrd2} we have
$ \frac{\partial  I}{\partial r}  >0$
so that $I_\theta$ is a strictly increasing function of $r$.
It is easy to see that $\lim_{r \to 0} I_\theta(r) = -\infty$. With \eqref{dilog1},
\begin{align*}
     \li(r e^{i \theta}) & = - \li \left(\frac 1r e^{-i \theta}\right) - \frac{\pi^2}6 -\frac 12 \log^2\left(\frac 1r e^{i(\pi- \theta)} \right)\\
     & = - \li\left(\frac 1r e^{-i \theta}\right) - \frac{\pi^2}6 -\frac 12 \Bigl(-\log r + i(\sgn \theta \cdot \pi - \theta) \Bigr)^2
\end{align*}
and hence
\begin{equation} \label{op}
    I_\theta(r)=- \Im \li\left(\frac 1r e^{-i \theta}\right) + (2\pi B + \sgn \theta \cdot \pi - \theta) \log r.
\end{equation}
It follows that
 $\lim_{r \to \infty} I_\theta(r) = \infty$ and so   exactly one  $r$ makes $I_\theta(r)=0$.
\end{proof}

%*************************************
% Graphics polar rect

\SpecialCoor
\psset{griddots=5,subgriddiv=0,gridlabels=0pt}
\psset{xunit=6cm, yunit=6cm, runit=6cm}
\psset{linewidth=1pt}
\psset{dotsize=3pt 0,dotstyle=*}

\begin{figure}[h]
\begin{center}
\begin{pspicture}(-0.3,-0.2)(1.2,0.6) %\psgrid

\savedata{\mydata}[
{{1.09412, -0.192923}, {1.08655, -0.158145}, {1.0769, -0.123882},
{1.06419, -0.0901737}, {1.04744, -0.0571767}, {1.02669, -0.0251928},
{0.993985, 0.00543425}, {0.978384, 0.034715}, {0.964928,
  0.0632614}, {0.953638, 0.0913186}, {0.94154, 0.118756}, {0.930639,
  0.145861}, {0.918965, 0.172419}, {0.907518, 0.198624}, {0.896302,
  0.224541}, {0.884352, 0.249965}, {0.871688, 0.274875}, {0.859274,
  0.299581}, {0.846165, 0.323791}, {0.833304, 0.347871}, {0.818851,
  0.371059}, {0.804658, 0.394134}, {0.78982, 0.416693}, {0.775226,
  0.439211}, {0.759138, 0.46071}, {0.743299, 0.482185}, {0.726857,
  0.503125}, {0.709829, 0.523514}, {0.692233, 0.543335}, {0.674088,
  0.562573}}
]

\savedata{\mydatb}[
{{0.0999888, 0.00149994}, {0.129974, 0.00259983}, {0.159946,
  0.00415953}, {0.189909, 0.00588906}, {0.219857,
  0.00791829}, {0.24978, 0.0104969}, {0.279691, 0.0131552}, {0.309565,
   0.0164223}, {0.339428, 0.0197089}, {0.369242, 0.0236638}, {0.39902,
   0.0279771}, {0.428759, 0.0326485}, {0.458416,
  0.0381362}, {0.488061, 0.0435525}, {0.517655, 0.0493257}, {0.547141,
   0.0560028}, {0.576558, 0.0630949}, {0.605901,
  0.0706014}, {0.635165, 0.0785217}, {0.664346, 0.0868549}, {0.693441,
   0.0956003}, {0.722339, 0.105479}, {0.751237, 0.115076}, {0.77991,
  0.125861}, {0.808592, 0.136304}, {0.837018, 0.147992}, {0.865306,
  0.160143}, {0.893624, 0.171862}, {0.921819, 0.183985}, {0.949886,
  0.19651}, {0.978031, 0.20846}, {1.00607, 0.220754}, {1.03422,
  0.232358}, {1.06228, 0.244248}, {1.09051, 0.255335}, {1.11866,
  0.266651}, {1.14674, 0.278196}}
]

\psline[linecolor=gray]{->}(-0.3,0)(1.3,0)
\psline[linecolor=gray]{->}(0,-0.2)(0,0.6)
\multirput(-0.2,-0.01)(0.1,0){15}{\psline[linecolor=gray](0,0)(0,0.02)}
\multirput(-0.01,-0.1)(0,0.1){7}{\psline[linecolor=gray](0,0)(0.02,0)}

%\dataplot[linecolor=orange,linewidth=0.8pt,plotstyle=line]{\mydatb}

\psarc[linecolor=orange,linestyle=dashed,dash=3pt 2pt](0,0){1}{350}{35}
\pscustom[linewidth=1pt,fillstyle=solid,fillcolor=orange]{
\psarc(0,0){0.8}{7.5}{22.5}
\psarcn(0,0){1}{22.5}{7.5}
\psarc(0,0){0.8}{7.5}{22.5}}
\psarc[linewidth=1pt,fillstyle=solid,fillcolor=white](0,0){0.8}{7.5}{22.5}

\dataplot[linecolor=blue,linewidth=0.8pt,plotstyle=line]{\mydata}
\dataplot[linecolor=blue,linewidth=0.8pt,plotstyle=line]{\mydatb}

\psdots[linecolor=blue,dotsize=5pt](0.916,0.182)

\rput(-0.07,0.5){$1/2$}
\rput(0.5,-0.07){$1/2$}
\rput(0.98,-0.07){$1$}
\rput(0.55,0.5){$I_\theta(r)=0$}
\rput(1.25,0.2){$R_r(\theta)=0$}
\rput(0.68,0.25){$\mathcal R_{0,1}$}

\end{pspicture}
\caption{The polar rectangle $\mathcal R_{0,1}$ containing the zero $w(0,1) \approx 0.916+0.182i$}\label{rfig}
\end{center}
\end{figure}
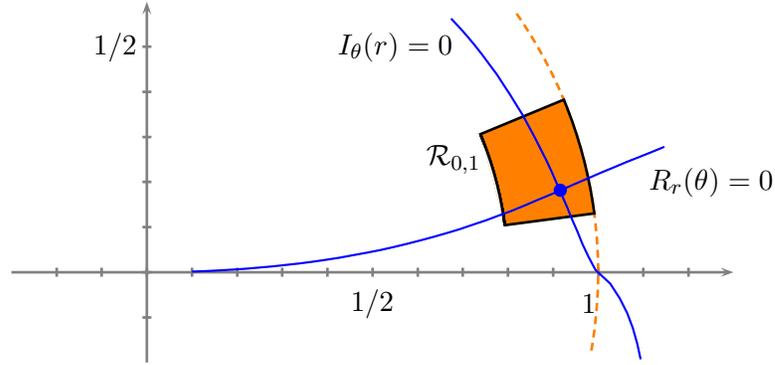
%*************************************

We may therefore define a function $g(\theta)=g_B(\theta)$, with domain $(-\pi,0) \cup (0,\pi)$, equalling the unique $r$ such that $I_\theta(r)=0$. This lets us parameterize the curve where $\Im\phi_{A,B}(z)=0$ for $z\in \C-\R$ as $(g(\theta),\theta)$ in polar coordinates. Recall $\kappa \approx 1.015$ defined in \eqref{kap}.

\begin{prop} \label{yu}
The function $g(\theta)$ is smooth and satisfies
\begin{alignat}{2}
    1 < & g(\theta) < \exp\left(\frac{\kappa}{\pi(2B-1)}\right) \qquad & & \text{ for } \quad \theta \in (-\pi,0), \label{b1}\\
   \exp\left(\frac{-\kappa}{\pi(2B)}\right) < & g(\theta) <1  \qquad & & \text{ for } \quad \theta \in (0,\pi). \label{b2}
\end{alignat}
\end{prop}
\begin{proof}
For any $\theta_0$ in the domain of $g$, set $r_0=g(\theta_0)$. Then $(r_0, \theta_0)$ is a solution to $I_\theta(r)=0$. By the implicit function theorem, $r$ satisfying $I_\theta(r)=0$ is a function of $\theta$  in a neighborhood of $\theta_0$ provided
\begin{equation} \label{alr}
   \left. \frac{\partial I}{\partial r}\right|_{(r,\theta)=(r_0,\theta_0)} \neq 0.
\end{equation}
As we have already seen with \eqref{plrd2}, the left side of \eqref{alr} is always $>0$. Thus, the implicit function theorem confirms that $r$ is a function of $\theta$, and that $g$ is differentiable as many times as $I$ is. Hence $g$ is a smooth function of $\theta$. By implicit differentiation we have for example, using \eqref{plrd2}, \eqref{plrd3},
\begin{equation}\label{imp}
    g'(\theta) = -\frac{\partial I}{\partial \theta}\Big/ \frac{\partial I}{\partial r} = r \frac{\log |1-r e^{i \theta}|}{2\pi B -  \arg(1-r e^{i \theta})} \qquad \quad \text{for $\quad r=g(\theta)$}.
\end{equation}

Next we prove the bounds \eqref{b1}, \eqref{b2}.
For $r=1$ we have $I_\theta(1)=\cl(\theta)$. Since we have seen that $I_\theta$ is a strictly increasing function of $r$ it follows that, see Figures \ref{bfig} and \ref{rfig},
\begin{alignat}{3}
    \theta & \in (-\pi,0) &  \quad \implies \cl(\theta) & <0  &  \quad \implies g(\theta) & > 1,\label{a1} \\
    \theta & \in (0,\pi) &  \quad \implies \cl(\theta) & >0  &  \quad \implies g(\theta) & < 1 \label{a2}
\end{alignat}
giving the lower bound in \eqref{b1} and the upper bound in \eqref{b2}.

For $\theta  \in (-\pi,0)$, $I_\theta(r)=0$ implies that $r>1$ by \eqref{a1} and, employing \eqref{op},
\begin{equation*}
    (2\pi B - \pi - \theta) \log r   = \Im \li\left(\frac 1r e^{-i \theta}\right)
         < \Im \li\left(1 e^{-i \theta}\right)
         \lqs \kappa
\end{equation*}
since $\Im \li\left(t e^{-i \theta}\right)$ is increasing in $t$ as we saw after \eqref{plrd}.
The upper bound in \eqref{b1} follows.
For $\theta  \in (0,\pi)$, $I_\theta(r)=0$ implies that $r<1$ by \eqref{a2} and so
\begin{equation*}
        2\pi   B  \log r  = -\Im \li(r e^{i \theta})
         > -\Im \li(1 e^{i \theta})
         \gqs -\kappa.
\end{equation*}
This gives the lower bound in \eqref{b2} and completes the proof.
\end{proof}

With Proposition \ref{yu}, we may restrict our attention to $r$ in the interval $(0.8,2)$ since
\begin{equation*}
    0.85 \approx \exp\left(\frac{-\kappa}{2\pi}\right) < r < \exp\left(\frac{\kappa}{\pi}\right) \approx 1.38
\end{equation*}
corresponding to \eqref{b1}, \eqref{b2} with $B=1$.

For $A$, $B$  fixed, write the real part of \eqref{mainp} as the function
\begin{equation} \label{plrxx}
    R_r(\theta):=\Re \li(r e^{i \theta}) +4 \pi^2 A -   2\pi   B  \theta   .
\end{equation}
We see with \eqref{wind2} that $R_r(\theta)$ is a continuous function of $\theta \in [-\pi, \pi]$. It is always smooth for $\theta \in (-\pi,0) \cup (0,\pi)$ but it is not differentiable at $\theta=0$ when $r\gqs 1$.
Similarly to \eqref{plrd2}, \eqref{plrd3} we have
\begin{align} \label{rr2}
 \frac{\partial R}{\partial \theta} & = -2\pi B +  \arg(1-r e^{i \theta}),\\
 \frac{\partial  R}{\partial r} & =  -  \frac 1r\log |1-r e^{i \theta}|. \label{rr3}
\end{align}

\begin{lemma}
Fix $B \in \Z_{\gqs 1}$ and $r$ in the interval $(0.8,2)$.
\begin{itemize}
\item If $-B/2<A \lqs B/2$, then  there exists a unique $\theta \in (-\pi,0) \cup (0,\pi)$ so that $R_r(\theta)=0$.
\item Otherwise, if $A \lqs -B/2$ or $ B/2<A$, then  there is no solution to $R_r(\theta)=0$ with $\theta \in (-\pi,0) \cup (0,\pi)$.
\end{itemize}
\end{lemma}
\begin{proof}
With \eqref{rr2} we have
$ \frac{\partial  R}{\partial \theta}  <0$ for $\theta \in (-\pi,0) \cup (0,\pi)$
so that $R_r$ is a strictly decreasing function of $\theta$. Therefore
\begin{equation*}
    R_r(\pi)< R_r(\theta) < R_r(-\pi) \quad \text{ for } \quad -\pi< \theta < \pi
\end{equation*}
where
\begin{align} \label{sr2}
 R_r(\pi) & = \Re \li(-r ) +4 \pi^2 A-   2\pi^2   B,  \\
 R_r(-\pi) & = \Re \li(-r ) +4 \pi^2 A +   2\pi^2   B  . \label{sr3}
\end{align}
Hence, for $r>0$ fixed, $R_r(\theta)=0$ has a solution $\theta$ if and only if $R_r(\pi)<0$ and $R_r(-\pi)>0$. With \eqref{sr2}, \eqref{sr3} these conditions are equivalent to
\begin{equation}\label{cond}
    -\frac{B}{2} < A + \frac{\Re \li(-r )}{4 \pi^2} < \frac{B}{2}.
\end{equation}
Note that $\Re \li(x)$ is increasing by \eqref{ddx} so that
\begin{equation*}
-1.44 \approx \Re \li(-2 ) < \Re \li(-r ) < \Re \li(-0.8 ) \approx -0.68,
\end{equation*}
and the result follows.
\end{proof}

For $A$, $B \in \Z$ with $-B/2<A\lqs B/2$, we may therefore define a function $h(r)=h_{A,B}(r)$, with domain $(0.8,2)$, equalling the unique $\theta \in (-\pi,0) \cup (0,\pi)$ such that $R_r(\theta)=0$. Then $(r,h(r))$ parameterizes a part of the curve defined by $\Re \, \phi_{A,B}(z)=0$.

\begin{prop} \label{yu2}
For $r\in (0.8,2)$ the function $h(r)=h_{A,B}(r)$ is smooth and satisfies
\begin{equation}\label{hb}
    \frac{\pi}B \left(2A-\frac 18 \right) < h_{A,B}(r) <  \frac{\pi}B\left(2A+\frac 18 \right) \qquad (-B/2<A<B/2).
\end{equation}
For $A=0$ we have the improvement on \eqref{hb}
\begin{equation}\label{hb3}
    \frac{\pi}{24B} < h_{0,B}(r) <  \frac{\pi}{8B}
\end{equation}
and for $A=B/2$ with $B$ even
\begin{equation}\label{hb2}
    \pi -\frac{\pi}{8B}  < h_{B/2,B}(r) <  \pi.
\end{equation}
\end{prop}
\begin{proof}
For any $r_0$ in the domain of $h$, set $\theta_0=h(r_0)$. Then $(r_0,\theta_0)$ is a solution to $R_r(\theta)=0$. By the implicit function theorem, $\theta$ satisfying $R_r(\theta)=0$ is a function of $r$  in a neighborhood of $r_0$ since
\begin{equation} \label{alr2}
   \left. \frac{\partial R}{\partial \theta}\right|_{(r,\theta)=(r_0,\theta_0)} < 0
\end{equation}
by \eqref{rr2}. The theorem also says that $h$ is differentiable as many times as $R$ is. Hence $h$ is a smooth function of $r$. By implicit differentiation we have for example, using \eqref{rr2}, \eqref{rr3},
\begin{equation}\label{imph}
    h'(r) = -\frac{\partial R}{\partial r}\Big/ \frac{\partial R}{\partial \theta} = -\frac 1r \frac{\log |1-r e^{i \theta}|}{2\pi B -  \arg(1-r e^{i \theta})}  \qquad \quad \text{for $\quad \theta=h(r)$}.
\end{equation}

The bounds \eqref{hb}, \eqref{hb3} and \eqref{hb2} follow from \eqref{plrxx} and  estimates for $\Re \li(r e^{i \theta})$ which we work out next.
First we note that $\Re \li(r e^{i \theta})$ is continuous for $r\gqs 0$ and $-\pi \lqs \theta \lqs \pi$. As in Lemma \ref{im0} we can show $\Re \li(r e^{-i \theta}) = \Re \li(r e^{i \theta})$. Also $\Re \li(r e^{i \theta})$ is strictly decreasing as a function of $\theta$ when $0\lqs \theta \lqs \pi$ and as a function of $r$ it is increasing for $r<2\cos \theta$ and decreasing for $r>2\cos \theta$.

The maximum value of $\Re \li(r e^{i \theta})$ occurs at $(r,\theta)=(2,0)$ and equals $\pi^2/4$. Since $\Re \li(-2)>-\pi^2/4$ it follows that
\begin{equation} \label{ebn}
    |\Re \li(r e^{i \theta})| < \pi^2/4 \qquad ( r < 2).
\end{equation}
Then \eqref{hb} is implied by \eqref{plrxx} and  \eqref{ebn}. For $A=0$ this means $-\frac{\pi}{8B} < h_{0,B}(r) <  \frac{\pi}{8B}$. However, for $r \in (0.8,2)$ and $-\pi/8<\theta<\pi/8$ we have $\Re \li(r e^{i \theta}) > \Re \li(0.8 e^{-i \pi/8})> \pi^2/12$. This gives the improved lower bound in \eqref{hb3}. Finally, \eqref{hb2} is a consequence of
\begin{equation}\label{web}
    \Re \li(r e^{i \theta})   <0     \quad \text{ for } \quad r>0, \ |\theta|  \gqs \pi/2.
\end{equation}
The bound \eqref{web} is true since $\Re \li(r e^{i \pi/2})$ is decreasing from the value $0$ at $r=0$.
\end{proof}

Let $B \in \Z_{\gqs 1}$. For each $A \in \Z$ satisfying $-B/2<A\lqs B/2$, based on the bounds in Propositions \ref{yu} and \ref{yu2}, define the following polar rectangles:
\begin{equation*}
    \mathcal R_{A,B}  := \Bigl\{ (r,\theta) \ \Big | \ r_1 \lqs r \lqs r_2, \ \theta_1 \lqs \theta \lqs  \theta_2 \Bigr\}
\end{equation*}
where
\begin{align*}
    r_1 = r_1(A,B) & := \begin{cases}1 & \text{\quad for \quad } -B/2<A<0 \\ \exp(-\kappa/(2\pi B)) & \text{\quad for \quad } 0 \lqs A \lqs B/2 \end{cases}\\
   r_2 =  r_2(A,B) & := \begin{cases}\exp(\kappa/(\pi (2B-1))) & \text{\quad for \quad } -B/2<A<0 \\ 1 & \text{\quad for \quad } 0 \lqs A \lqs B/2 \end{cases} \\
    \theta_1 = \theta_1(A,B) & := \begin{cases}\pi(2A-1/8)/B & \text{\quad for \quad } A \neq 0 \\ \pi/(24B) & \text{\quad for \quad } A=0 \end{cases}\\
   \theta_2 = \theta_2(A,B) & := \begin{cases}\pi(2A+1/8)/B & \text{\quad for \quad } A \neq B/2 \\ \pi & \text{\quad for \quad } A=B/2. \end{cases}
\end{align*}
For example, see Figure \ref{rfig} for $\mathcal R_{0,1}$.

\begin{theorem} \label{cal}
For $A \in \Z$ and $B \in \Z_{\gqs 1}$,  we have solutions to $\phi_{A,B}(z)=0$ if and only if $-B/2<A\lqs B/2$. For such a pair $A,B$ the solution $z$ is unique and contained in the interior of the polar rectangle $\mathcal R_{A,B}$.
\end{theorem}
\begin{proof}
Let $r_1$, $r_2$, $\theta_1$, $\theta_2$ be defined as above, giving the boundaries of $\mathcal R_{A,B}$. With Propositions \ref{yu} and \ref{yu2} we have seen that there exist continuously differentiable functions $g$, $h$ with
\begin{equation}\label{doma}
    g:[\theta_1,\theta_2] \to (r_1,r_2), \qquad h:[r_1,r_2] \to (\theta_1,\theta_2)
\end{equation}
and
\begin{equation} \label{drvs}
    g'(\theta) =  r \cdot \frac{\log |1-r e^{i \theta}|}{2\pi B -  \arg(1-r e^{i \theta})}, \qquad h'(r) = -\frac 1r \cdot \frac{\log |1-r e^{i \theta}|}{2\pi B -  \arg(1-r e^{i \theta})}.
\end{equation}
Solutions to $\phi_{A,B}(z)=0$ must be points in the intersection
\begin{equation}\label{alcu}
(g(\theta),\theta)_{\theta_1\lqs \theta \lqs \theta_2} \cap (r,h(r))_{r_1\lqs r\lqs r_2}.
\end{equation}
The denominators $2\pi B -  \arg(1-r e^{i \theta})$ in \eqref{drvs} are positive and bounded below by $\pi(2B-1)$. We have $\log |1-r e^{i \theta}|=0$ if and only if $r=2\cos \theta$ and it is straightforward to show that the curve $r=2\cos \theta$ intersects $\mathcal R_{A,B}$ if and only if $A=\pm B/6$.

Suppose first that $-B/6<A<B/6$. Then we have $\log |1-r e^{i \theta}|<0$ for all  $(r,\theta) \in \mathcal R_{A,B}$ and therefore  $h'(r)>0$ and $g'(\theta)<0$ for all $r$ and $\theta$   in the domains \eqref{doma}. By the inverse function theorem, $g$ has an inverse, $g^{-1}(r)$, a strictly decreasing continuously differentiable function on $[g(\theta_1),g(\theta_2)]$. Set $f(\theta)$ to be the difference $h(r)-g^{-1}(r)$. Then $f$ is continuous and strictly increasing on $[g(\theta_1),g(\theta_2)]$ with $f(g(\theta_1))<0$ and $f(g(\theta_2))>0$. By the intermediate value theorem, there is a unique $r^*$ so that $f(r^*)=0$. Set $\theta^*=h(r^*)$. Then $(r^*,\theta^*)$ is the unique element of \eqref{alcu} and it lies in the interior of $\mathcal R_{A,B}$.

The cases with $A<-B/6$ or $A>B/6$ are handled very similarly to the above, the only difference being that now $g'(\theta)>0$ and $h'(r)<0$.

For the last case we have $A=\pm B/6$.
 With \eqref{drvs} we may easily show that $|g'(\theta)|$, $|h'(r)|<1$. Since $h'$ may be zero, it does not necessarily have an inverse.
 Consider $\mathcal R_{A,B}$ as a rectangle in the $r$ $\theta$ plane and rotate it about the origin, say, by an angle of $\pi/4$. The curves corresponding to the rotated curves $(r,h(r))$ and $(g(\theta),\theta)$ may now be expressed as graphs of functions (of positive and negative slope, respectively) and our previous argument applies.
\end{proof}

Theorem \ref{cal}  establishes the $B \neq 0$ case of Theorem \ref{ebt} (except for simplicity) and shows that when $\phi_{A,B}$ has a zero it is close to $e^{2\pi i A/B}$. In the next section, we find each zero more precisely.

\section{Newton's method for $\phi_{A,B}$} \label{66}

\begin{prop} \label{d1}
Let $L$  be a line in $\C$ that is a distance $r>0$ from the origin. Suppose $z$ and $c$ are two points  in the half plane that is bounded by $L$ and does not contain the origin. If the ray $(-\infty,0]$ intersects this half plane, we also assume $z$ and $c$ are on the same side of $(-\infty,0]$.
With these assumptions
\begin{equation*}
    \left| \log (z) - \log (c) - \frac{z-c}{c}\right| \lqs  |z-c|^2 \frac{1}{2r^2}.
\end{equation*}
\end{prop}
\begin{proof}
As in Proposition \ref{d0}, express $\log(z)$ with the first two terms of its Taylor expansion at $c$ to get
\begin{equation*}
    \log (z) = \log (c) +\log'(c)(z-c) + \frac{(z-c)^2}{2\pi i}\int_C \frac{\log(w)}{(w-c)^2(w-z)} \, dw
\end{equation*}
with $C$ a positively oriented circular curve of radius $T$ containing $z$ and $c$, similar to Figure \ref{cir} but this time avoiding the branch cut $(-\infty,0]$. Let the horizontal paths above and below the cut coincide. The  values of $\log$  on these paths differ by $2\pi i$, as in \eqref{wind}. We now rotate these horizontal paths away from $z$ and $c$, (using values of the continued log that keep the difference on the two paths as $2\pi i$),
until they are perpendicular to $L$, not intersecting it. To keep track of the angle, suppose these paths now pass through $\alpha$  with $|\alpha|=1$. Letting $T \to \infty$ and $\varepsilon \to 0$ we find
\begin{equation} \label{hob}
    \frac{1}{2\pi i}\int_C \frac{\log(w)}{(w-c)^2(w-z)} \, dw = -\int_{0}^\infty \frac{\alpha}{(\alpha t-c)^2(\alpha t-z)} \, dt.
\end{equation}
Since $|\alpha t-c|$, $|\alpha t-z|  \gqs t+r$, the right side of \eqref{hob} is bounded in absolute value by $\int_{0}^\infty (t+r)^{-3} \, dt =1/(2r^2)$ as required.
\end{proof}

Define
\begin{equation*}
   M(s):=\left(\frac{2\log(4/3)}{s}+ 8\pi\right)\frac{1}{(4s+1)^2}+3.
\end{equation*}

\begin{prop} \label{d2}
Let $L$  be a line in $\C$ that is a distance $s>0$ from the point $1$. Suppose $z$ and $c$ are two points  in the half plane that is bounded by $L$ and does not contain $1$. If the ray $[1,\infty)$ intersects this half plane, we assume $z$ and $c$ are on the same side of $[1,\infty)$.
With these assumptions
\begin{equation*}
    \left| \li (z) - \li (c) - (z-c) \lip (c)\right| \lqs |z-c|^2 M(s).
    %\frac{\pi}{2s^2}+ \frac{s\log(s)-s+1}{2s(s-1)^2}
\end{equation*}
\end{prop}
\begin{proof}
As in Propositions \ref{d0} and \ref{d1} we must bound the remainder term in
\begin{equation*}
    \li (z) = \li (c) + (z-c) \lip (c) + \frac{(z-c)^2}{2\pi i}\int_C \frac{\li(w)}{(w-c)^2(w-z)} \, dw
\end{equation*}
where  $C$ is the circular path of radius $T$ containing $z$ and $c$ in Figure \ref{cir}. Let the paths above and below the branch cut $[1,\infty)$ coincide - the difference between values of $\li(w)$ with $w$ coming from above and below the branch cut  is $2\pi i \log(w)$ by \eqref{wind2}. Similarly to Proposition  \ref{d1}, we rotate these horizontal paths away from $z$ and $c$ until they are perpendicular to $L$, using values of the continued dilogarithm that keep the difference at a point $w$ on the two paths as $2\pi i \log(w)$. Suppose these paths now pass through $1+\alpha$  with $|\alpha|=1$.
Letting $T \to \infty$ and $\varepsilon \to 0$ we find
\begin{equation*}
    \frac{1}{2\pi i}\int_C \frac{\li(w)}{(w-c)^2(w-z)} \, dw = \int_{0}^\infty \frac{\alpha \log(\alpha t +1)}{(\alpha t+1-c)^2(\alpha t+1-z)} \, dt.
\end{equation*}
Then $|\alpha t+1-c|$, $|\alpha t+1-z|  \gqs t+s$ imply
\begin{equation} \label{hob2}
    \left|\int_{0}^\infty \frac{\alpha \log(\alpha t +1)}{(\alpha t+1-c)^2(\alpha t+1-z)} \, dt \right| \lqs \int_{0}^\infty \frac{|\log(\alpha t +1)|}{(t+s)^3} \, dt.
\end{equation}
With the straightforward inequality
\begin{equation}\label{sinq}
    \left|\log(z+1)\right| \lqs  \frac{1}{Y}\log\left(\frac{1}{1-Y}\right) \cdot |z| \qquad (|z|\lqs Y<1)
\end{equation}
we have $|\log(\alpha t +1)| \lqs  4\log(4/3) \cdot t$ for $|t|\lqs 1/4$. Also
\begin{equation}\label{cru}
    | \log(\alpha t +1)| \lqs |\Im \log(\alpha t +1)|+ |\Re \log(\alpha t +1)| \lqs \pi + \log(t +1),
\end{equation}
so the right side of \eqref{hob2} is bounded  by
\begin{multline}\label{cru2}
\int_{0}^{1/4} \frac{ 4\log(4/3) t}{(t+s)^3} \, dt + \int_{1/4}^\infty \frac{ \pi}{(t+s)^3} \, dt
+ \int_{1/4}^\infty \frac{\log(t+1)}{(t+s)^3} \, dt \\
\lqs \frac{2\log(4/3)}{s(4s+1)^2}+\frac{8\pi}{(4s+1)^2}+\int_{1/4}^\infty \frac{\log(t+1)}{t^3} \, dt.
\end{multline}
The last integral on the right of \eqref{cru2} equals $2+\frac{15}{2}\log(5)-16\log(2)<3$.
\end{proof}

Let $A$, $B$ be integers  with $-B/2<A\lqs B/2$.
We know that $\phi_{A,B}(\rho)=0$ for some unique $\rho \in \mathcal R_{A,B}$. Let $\mathcal R'_{A,B}$ be a convex version of $\mathcal R_{A,B}$ with the boundary arc of radius $r_1(A,B)$ replaced by a straight line between the corners.
Let $c$ be any point in   $\mathcal R'_{A,B}$. Then
\begin{equation} \label{fder}
    |c \cdot \phi'_{A,B}(c)|=|-\log(1-c)+2\pi i B| \gqs \pi(2B-1)
\end{equation}
and in particular $\phi'_{A,B}(c) \neq 0$, so that  $c^*:=c-\phi_{A,B}(c)/\phi'_{A,B}(c)$ makes sense.

\begin{theorem} \label{nwt}
Let $A$, $B$ be integers  with $-B/2<A\lqs B/2$.
If $\rho$, $c \in \mathcal R'_{A,B}$ with $\phi_{A,B}(\rho)=0$, then
\begin{equation} \label{mon}
    \left| \rho - c^*\right| < B |\rho-c|^2 \times
    \begin{cases} 0.76 & \text{if} \quad A \neq 0\\
    2.51 & \text{if} \quad A = 0.
    \end{cases}
\end{equation}
\end{theorem}
\begin{proof}
Suppose all points in $\mathcal R'_{A,B}$ are at least a distance $r$ from $0$ and $s$ from $1$. Then  Propositions \ref{d1} and \ref{d2} imply
\begin{equation} \label{ren}
    \left| \phi_{A,B}(\rho) - \phi_{A,B}(c) - (z-c) \phi'_{A,B}(c)\right| \lqs |\rho-c|^2 \left( \pi B/r^2 +M(s)\right).
\end{equation}
Combining \eqref{ren} with \eqref{fder} shows
\begin{equation}
    \left| \rho - c^*\right| \lqs |\rho-c|^2 \cdot B \cdot N(A,B) \label{cor}
\end{equation}
for
\begin{equation} \label{nab}
    N(A,B) :=  r_2(A,B) \frac{ \pi B/r^2+M(s)}{\pi B(2B-1)}
\end{equation}
since $|c| \lqs r_2(A,B)$.
Looking at the easier cases with $A \neq 0$ and $B\gqs 2$ first, it is routine to show
\begin{equation} \label{vals}
     r_2(A,B) < 6/5, \quad r>9/10, \quad s>2/B.
\end{equation}
With these numbers, for $A\neq 0$,
\begin{equation} \label{valsb}
    N(A,B) < \frac{6}{5\pi (2B-1)}\left( \frac{100\pi}{81}+\left[\log\left(\frac 43\right) +\frac{8\pi}{B}\right]\left(\frac{B}{B+8}\right)^2+\frac{3}B\right).
\end{equation}
Then $N(A,2)<0.76$, $N(A,3)<0.43$, $N(A,4)<0.30$ and if we replace $\frac{B}{B+8}$ by $1$ in \eqref{valsb} we see that $N(A,B)<0.42$ for all $B\gqs 5$. This proves the theorem for $A \neq 0$.

To treat the cases with $A=0$ we need to rework Proposition \ref{d2} a little. The set $\mathcal R'_{0,B}$ has all its points at least $\sin(\pi/(24B)>1/(8B)$ vertically above $\R$. Taking $\alpha =-i$ in Proposition \ref{d2} means we need to bound $|\log(-it+1)|=|\log(it+1)|$ in \eqref{hob2}.

\begin{lemma}
We have $|\log(it+1)| \lqs t$ for all $t\gqs 0$.
\end{lemma}
\begin{proof}
Let $\theta = \arg(it+1)$ and our desired inequality is equivalent to
\begin{equation*}
    g(\theta):=\tan^2 \theta-\theta^2-\log^2(\cos \theta) \gqs 0 \qquad (0\lqs \theta <\pi/2).
\end{equation*}
With \cite[(11.3)]{Ra}, write
\begin{equation*}
    \tan \theta = \frac 1\theta\sum_{n=2}^\infty a_n \theta^n \qquad \text{for} \qquad a_n= \begin{cases} 2^n(2^n-1)|B_n|/n!>0 & \text{if $n$ is even}\\
    0 & \text{if $n$ is odd.}
    \end{cases}
\end{equation*}
This expansion is valid for $|\theta|<\pi/2$. Therefore
\begin{equation} \label{comp}
    \tan^2 \theta =\frac 1{\theta^2}\sum_{n=4}^\infty \theta^n \sum_{j=0}^n a_j a_{n-j}.
\end{equation}
Since $\log(\cos \theta) = \sum_{n=2}^\infty \frac{a_n}{n+1} \theta^n$ we find
\begin{equation} \label{comp2}
    \log^2(\cos \theta) =\sum_{n=4}^\infty \theta^n \sum_{j=0}^n \frac{a_j a_{n-j}}{(j+1)(n-j+1)}.
\end{equation}
Comparing \eqref{comp} and \eqref{comp2} shows
\begin{equation*}
    \tan^2 \theta \gqs \frac{5 \log^2(\cos \theta)}{\theta^2}
\end{equation*}
and so
\begin{equation*}
    g(\theta) \gqs -\theta^2+ \left(\frac{5}{\theta^2}-1 \right)\log^2(\cos \theta).
\end{equation*}
Since $\log^2(\cos \theta) = \theta^4/4+\cdots$ it follows that $g(\theta)\gqs 0$ for $0\lqs \theta \lqs 1$. When $1\lqs \theta <\pi/2$ it is straightforward to verify that $g(\theta)\gqs 0$ since $g'(\theta)>0$ in this range.
\end{proof}

Now the remainder term in \eqref{hob2} is
\begin{equation} \label{bol}
    \int_{0}^\infty \frac{|\log(-i t +1)|}{(t+s)^3} \, dt \lqs \int_{0}^\infty \frac{t}{(t+s)^3} \, dt = \frac{1}{2s}.
\end{equation}
The quantities we need for $\mathcal R'_{0,B}$  satisfy
 \begin{equation} \label{vals2}
     \quad r_2(0,B) = 1, \quad r>9/10, \quad s>1/(8B).
\end{equation}
Then \eqref{cor} is true   with $M(s)$ in \eqref{nab} replaced by $1/(2s)$ from \eqref{bol}. We obtain
\begin{equation} \label{nab2}
    N(0,B)< \left( \frac{100\pi}{81}+4\right)\frac{1}{\pi(2B-1)}
\end{equation}
so that  $N(0,B) \lqs N(0,1) < 2.51$ for $B\gqs 1$. This completes the proof of Theorem \ref{nwt}.
\end{proof}

%Recall the sequence defined by Newton's method in \eqref{roov}.

\begin{theorem} \label{new}
Let $A$ and $B$ be integers satisfying $-B/2<A\lqs B/2$.  Let $\rho$ be the unique zero of $\phi_{A,B}$. Then the sequence $c_0$, $c_1 , \cdots$ in \eqref{roov} converges  to $\rho$ with
\begin{equation} \label{rap}
    |\rho-c_n| <1.25 (0.95)^{2^n}.
\end{equation}
\end{theorem}
\begin{proof} Suppose first that $A \neq 0$.
The distance between any two points in $\mathcal R'_{A,B}$ may be  shown to be less than $6/(5B)$. It then follows from Theorem \ref{nwt} that for any $c \in \mathcal R'_{A,B}$ we have $|\rho-c^*|< |\rho-c|$ so that $c^*$ remains in $\mathcal R'_{A,B}$.
We have $c_0 \in \mathcal R'_{A,B}$ and repeated applications of Theorem \ref{nwt} show that
\begin{equation*}
    0.8 B |\rho-c_n| < \bigl(0.8 B |\rho-c_{n-1}|\bigr)^2 < \bigl(0.8 B |\rho-c_{n-2}|\bigr)^4 <
    \cdots < \bigl(0.8 B |\rho-c_{0}|\bigr)^{2^n}.
\end{equation*}
The inequality \eqref{rap} now follows using  $|\rho-c_0|<6/(5B)$.

The case with $A=0$ is proved similarly. Use that the distance between any two points in $\mathcal R'_{0,B}$ is less than $1/(3B)$.
\end{proof}

\section{Further results} \label{77}

\begin{lemma} \label{66lem}
We have
\begin{equation*}
    \log(1-e^{i \theta})= -\clp(\theta)+
    \frac{i}{2} \left(\theta-\pi \frac{\theta}{|\theta|}\right)  \qquad (0<|\theta|<\pi ).
\end{equation*}
\end{lemma}
\begin{proof}
We may write $1-e^{i \theta}=-e^{i \theta/2} 2 i \sin (\theta/2) = e^{i (\theta-\pi)/2} 2 \sin (\theta/2)$. Taking logs with appropriate branches and using $\clp(\theta)=-\log | 2 \sin(\theta/2)|$ from \eqref{byo} completes the proof.
\end{proof}

\begin{lemma} \label{lm}
For $0<|\theta| \lqs \pi$
\begin{equation*}
    \left| \clp(\theta)\right| \lqs \log \left(\max \left\{ 2, \ \frac{\pi}{3|\theta|}\right\}\right).
\end{equation*}
\end{lemma}
\begin{proof}
When $\pi/3 \lqs |\theta| \lqs \pi$ we have $1 \lqs |2\sin(\theta/2)| \lqs 2$ so that $|\log | 2\sin(\theta/2)|| \lqs \log 2$.
When $0 < |\theta| < \pi/3$ we have $|2\sin(\theta/2)|<1$ and using the  inequality
$|\sin x| \gqs 3 |x|/\pi$ for $|x| \lqs \pi/6$   implies
\begin{equation*}
    \bigl|\log | 2\sin(\theta/2)|\bigr|= \log\left(\frac 1{|2\sin(\theta/2)|} \right) \lqs \log \frac{\pi}{3|\theta|}. \qedhere
\end{equation*}
\end{proof}

\begin{theorem} \label{new3}
Let $A$ and $B$ be integers satisfying $-B/2<A\lqs B/2$ and suppose $\phi_{A,B}(\rho)=0$. Then
\begin{equation} \label{rag}
    \rho=e^{2\pi i A/B}\left( 1+\frac{-\cl(2\pi A/B)+i \pi^2 B_2(|A|/B)}{2\pi B}\right) +O\left(\frac{1+\log B}{B^2}\right).
\end{equation}
\end{theorem}
\begin{proof}
Let $c_0$ and $c_1$ be as in Theorem \ref{new}. With \eqref{valsb} and \eqref{nab2} it is easy to see that $B \cdot N(A,B)$ is bounded by an absolute constant. Also $|\rho-c_0|<6/(5B)$, as we have seen in the proof of Theorem \ref{new}, so that \eqref{cor} implies
\begin{equation} \label{w36}
    |\rho - c_1| =O\bigl(B^{-2}\bigr).
\end{equation}
Now for $A \neq 0$,
\begin{equation*}
    \phi_{A,B}(c_0)  = -\li(e^{2\pi i A/B}) =  \pi^2 B_2(|A|/B) + i\cl(2\pi A/B)
\end{equation*}
with \eqref{imc2} and \eqref{imc}. Also
\begin{equation*}
    c_0 \phi'_{A,B}(c_0)  = 2 \pi i B -\log(1-e^{2\pi i A/B}) = 2\pi iB+ \clp(2\pi A/B) + 2\pi i\left( -\frac{A}{2B}+\frac A{4 |A|}\right)
\end{equation*}
using Lemma \ref{66lem}.
Therefore
\begin{equation} \label{u2}
     -\frac{\phi_{A,B}(c_0)}{c_0 \phi'_{A,B}(c_0)} = \frac{-\cl(2\pi A/B)+i \pi^2 B_2(|A|/B)}{2\pi B}\frac 1{1+X}
\end{equation}
for
\begin{equation*}
    X:=\frac 1B \left(  -\frac{A}{2B}+\frac A{4 |A|}\right)+ \frac{\clp(2\pi A/B)}{2\pi i B}.
\end{equation*}
We have
\begin{equation} \label{clp}
    \left| \clp(2\pi A/B)\right| \lqs \log(2B)
\end{equation}
in an easy corollary to Lemma \ref{lm}.
Use \eqref{clp} to see that
 \begin{equation} \label{u22}
    |X| \lqs \frac{\pi/2+\log(2B)}{2\pi B} \lqs \frac 12 \quad \implies \quad
    \left| \frac{1}{1+X} -1 \right| \lqs2|X| \lqs \frac {\pi +2\log(2B)}{2\pi B}
 \end{equation}
 and hence \eqref{u2} and \eqref{u22} together imply
\begin{equation}\label{xbx}
    1-\frac{\phi_{A,B}(c_0)}{c_0 \phi'_{A,B}(c_0)} = 1+\frac{-\cl(2\pi A/B)+i \pi^2 B_2(|A|/B)}{2\pi B} + O\left(\frac{1+\log B}{B^2}\right).
\end{equation}
Multiplying both sides of \eqref{xbx} by $c_0$ and using \eqref{w36} completes the proof of \eqref{rag}.

In the case $A=0$, our goal \eqref{rag} becomes $\rho=1+\pi i/(12B)+O\bigl((1+\log B)/B^2\bigr)$.
Similarly to \eqref{xbx} we find
\begin{equation*}
    1-\frac{\phi_{0,B}(c_0)}{c_0 \phi'_{0,B}(c_0)} = 1 + O\left(\frac{1+\log B}{B^2}\right)
\end{equation*}
by using \eqref{dtb}. Multiplying by $c_0$ then shows
\begin{equation} \label{xbx2}
    c_1=e^{\pi i/(12B)}+ O\left(\frac{1+\log B}{B^2}\right).
\end{equation}
Finally, \eqref{xbx2} and \eqref{w36} prove \eqref{rag}.
\end{proof}

Theorems \ref{new} and \ref{new3} above establish Theorem \ref{new2}.
The following result completes the last part of the proof of Theorem \ref{ebt}.

\begin{prop} \label{simple}
For $A$ and $B\in \Z$, all zeros of $\phi_{A,B}(z)$ are simple.
\end{prop}
\begin{proof}
The zero of $\li(z)=\phi_{0,0}(z)$ at $z=0$ is clearly simple by \eqref{def0}. For $B=0$ and $A \gqs 1$,  we have by \eqref{dphi} and Proposition \ref{rhoa} that
\begin{equation*}
    \rho \cdot \phi'_{A,0}(\rho) = -\log(1-\rho) <  -\pi\sqrt{8A-1/3} <0,
\end{equation*}
which shows that $\rho$ is simple. For $-B/2<A\lqs B/2$, if $\phi'_{A,B}(\rho) = 0$ then $2\pi i B = \log(1-\rho)$ by \eqref{dphi}. But this is impossible since $-\pi<\Im \log(1-\rho) \lqs \pi$.
\end{proof}

Another way to verify the simplicity of the zeros of $\phi_{A,B}(z)$, as well as their existence and uniqueness, is with the argument principle. Let $C$ be the contour shown in Figure \ref{cir}, but also avoiding the branch cut $(-\infty,0]$ if $B \neq 0$. Then
for $A$, $B \in \Z$
\begin{equation*}
    \Psi_C(A,B):= \frac{1}{2\pi i}\int_C \frac{\phi'_{A,B}(z)}{\phi_{A,B}(z)} \, dz
\end{equation*}
counts the number of zeros of $\phi_{A,B}(z)$ in the interior of $C$ with multiplicity. Letting the horizontal paths coincide, the large radius of $C$ go to infinity and the small radii go to zero,  we define $\Psi(A,B)$ which counts all the zeros of  $\phi_{A,B}(z)$. With \eqref{wind} and \eqref{wind2} we obtain
\begin{equation}\label{coni}
    2\pi i \Psi(A,B)= \int_1^\infty \left(
    \frac{\phi'_{A,B+1}(z)}{\phi_{A,B+1}(z)} - \frac{\phi'_{A,B}(z)}{\phi_{A,B}(z)}
    \right)\, dz
    +  \int_{-\infty}^0 \left(
    \frac{\phi'_{A,B}(z)}{\phi_{A,B}(z)} - \frac{\phi'_{A+B,B}(z)}{\phi_{A+B,B}(z)}
    \right)\, dz.
\end{equation}
Note that $\Psi(A,B)$ correctly counts the zeros of $\phi_{A,B}(z)$ when $B=0$; in that case there is no branch cut at $(-\infty,0]$ to avoid and the second integral in \eqref{coni} is zero.
Theorem \ref{ebt} implies
\begin{equation}\label{resu}
    \Psi(A,B) =   \begin{cases}
    1 & \text{\ if \ }B=0, \ A \gqs 0\\
    1 & \text{\ if \ }-|B|/2<A\lqs |B|/2\\
    0 & \text{\ otherwise}
    \end{cases}
\end{equation}
with the right side of \eqref{resu} indicating a single simple zero with $1$ and no zero with $0$.
Checking formula \eqref{resu}  numerically, we confirm that it is true for $|A|$, $|B| \lqs 50$. Perhaps \eqref{resu} can be proved directly from \eqref{coni}.

\section{Zeros of polylogarithms} \label{polyl}
The dilogarithm $\li(z)$ is a special case of the polylogarithm, also known as Jonqui\`ere's function \cite{Jon},
\begin{equation*}
    \pl_s(z):=\sum_{n=1}^\infty \frac{z^n}{n^s} \quad \text{ for } \quad |z|< 1, \ s\in \C.
\end{equation*}
As a function of $z$, it has an analytic continuation to all of $\C$ and in general will be multi-valued with branch points at $0$, $1$ and $\infty$, see for example Sections 4, 11 of \cite{Ve}. It satisfies
\begin{equation} \label{derli}
    z \frac{d}{dz}\pl_s(z) = \pl_{s-1}(z).
\end{equation}

\subsection{Zeros of $\pl_{s}(z)$ for   $\Re(s) > 0$} \label{lrr}
Here we discuss what is known about the zeros of $\pl_s(z)$ for $s \in \C$ with $\Re(s) > 0$.
%All these results are for $z$ in the principal domain $\C - [1,\infty)$.

\begin{theorem}[Le Roy, 1900 \cite{Ler}]
For $r > 0$, $\pl_r(z)$ has exactly one  zero for $z \in \C - [1,\infty)$. It is at $z=0$ and is simple.
\end{theorem}
\begin{proof}
We use Jonqui\`ere's representation \cite{Jon}
\begin{equation}\label{rere}
    \pl_{s}(z)=\frac{z}{\G(s)}\int_0^\infty \frac{t^{s-1}}{e^t-z}\, dt \qquad (\Re(s)>0, \ z \in \C-[1,\infty)).
\end{equation}
Then for $s=r>0$ and $z=x+i y$, the imaginary part of $\pl_r(z)/z$ is
\begin{equation*}
    \frac{y}{\G(r)}\int_0^\infty \frac{t^{r-1}}{(e^t-x)^2+y^2} \, dt
\end{equation*}
and therefore non-zero for $y \neq 0$. For $y=0$ we have $z=x<0$ and so, clearly, $\pl_r(x)/x>0$. We have shown that $\pl_r(z)/z$ is finite and non-zero for $z \in \C - [1,\infty)$ as required.
\end{proof}

For $s=2$ we have a clear picture of the zeros of $\li(z)$ with Theorems \ref{ebt} -- \ref{new2}. When  $s=1$ we have the simple case $\pl_{1}(z)=-\log(1-z)$.
Clearly $\pl_{1}(0)=0$  gives the only zero on the principal branch and $\pl_{1}(z)$ is non-zero on every other branch.
Going in the other direction,  if we let $\pl_{3}(z)$ denote the trilogarithm on its principal branch, it may be shown (see \cite[p. 246]{Ve}) that on any branch it has the form
\begin{equation}\label{tril}
\pl_{3}(z) +  4\pi^3 i A+  2\pi^2  B  \log(z)- \pi i  C \log^2(z) \qquad (A,B,C \in \Z).
\end{equation}
By studying when the real and imaginary parts of \eqref{tril} vanish, as in Section \ref{55},  it should be possible to approximately determine the zeros. Looking at some cases, it seems there may be up to two zeros on each branch and these zeros are close to the unit circle.

For general $s$ with $\Re(s)>0$, the only result on the zeros of $\pl_{s}(z)$ seems to be that they are finite in number \cite{G79} for $z$ on the principal branch. Vep\v{s}tas  gives an efficient method to compute polylogarithms in   \cite{Ve} and displays the zeros of $\pl_{1/2+80i}(z)$, $\pl_{1/2+15i}(z)$ and   $\pl_{6/5+14i}(z)$ in the phase plots \cite[Figs 8-10]{Ve}. In these cases the zeros lie near the unit circle. If $s_n$ is a zero of the Riemann zeta function $\zeta(s)$ then $\pl_{s_n}(z)$ has zeros at $z=\pm 1$. As $s$ moves continuously near $s_n$ the corresponding $z$ zeros of $\pl_{s}(z)$ show interesting behavior moving near $\pm 1$. This connection is explored in \cite{FK}.

For another example, Figure \ref{xfig} shows the zeros of $\pl_{s}(z)$ for $s=10+44i$. The zeros were found numerically by combining a phase plot with Newton's method. The spiraling curve that these zeros are making seems to be of a similar form to \eqref{plzo}. It would be interesting to identify it exactly.

%*************************************
% Graphics polylog zeros s=10+44i

\SpecialCoor
\psset{griddots=5,subgriddiv=0,gridlabels=0pt}
\psset{xunit=0.3cm, yunit=0.2cm}
\psset{linewidth=1pt}
\psset{dotsize=3pt 0,dotstyle=*}

\begin{figure}[h]
\begin{center}
\begin{pspicture}(-10,-15)(20,6.5) %\psgrid

\savedata{\mydata}[
{{0., 0.}, {21.1251, -6.7895}, {4.7183, -15.2882}, {-5.4338,
-10.5751}, {-8.2936, -3.9275}, {-7.2441, 0.7695}, {-4.9906,
  3.1884}, {-2.8248, 4.0277}, {-1.1539, 4.0013}, {-0.0469,
  3.5765}, {0.6806, 3.0675}, {1.1849, 2.5101}, {1.4337,
  2.0347}, {1.5398, 1.6725}, {1.6061, 1.3056}, {1.6079,
  1.028}, {1.5719, 0.834}, {1.531, 0.6668}, {1.4844, 0.534}, {1.4367,
  0.4284}, {1.3905, 0.3442}, {1.3471, 0.2769}, {1.3073,
  0.223}, {1.2713, 0.1796}, {1.239, 0.1447}, {1.2101, 0.1165}, {1.186,
   0.088}, {1.162, 0.074}, {1.142, 0.059}, {1.123, 0.043}, {1.107,
  0.033}, {1.086, 0.026}}
]

\psline[linecolor=gray]{->}(-11,0)(22,0)
\psline[linecolor=gray]{->}(0,-15)(0,6)
\multirput(-10,-0.5)(2,0){16}{\psline[linecolor=gray](0,0)(0,1)}
\multirput(-0.3,-14)(0,2){10}{\psline[linecolor=gray](0,0)(0.6,0)}

%\dataplot[linecolor=orange,linewidth=0.8pt,plotstyle=line]{\mydatb}
\dataplot[linecolor=blue,linewidth=0.8pt,plotstyle=dots]{\mydata}

\psdots[linecolor=blue,dotsize=5pt](0,0)

\rput(-10,-1.5){$-10$}
\rput(10,-1.5){$10$}
\rput(20,-1.4){$20$}
\rput(-1.7,-10){$-10$}

\end{pspicture}
\caption{The zeros of $\pl_{s}(z)$ for $s=10+44i$}\label{xfig}
\end{center}
\end{figure}
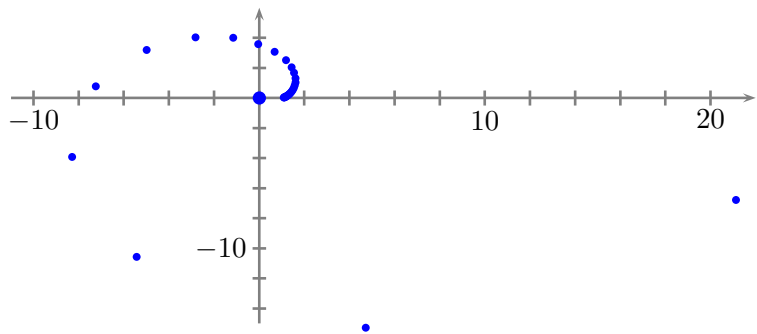
%*************************************

\subsection{Zeros of $\pl_{s}(z)$ for   $\Re(s) \lqs 0$}
Much more is known about the zeros of $\pl_{s}(z)$ for   $\Re(s) \lqs 0$. We look at the case when $s$ is an integer first.
Applying \eqref{derli} to $\pl_{1}(z)$ shows $\pl_{0}(z)=z/(1-z)$, a rational function.
For $-m \in \Z_{\lqs 0}$ write
\begin{equation} \label{fro}
    \pl_{-m}(z) = \frac{z \cdot A_m(z)}{(1-z)^{m+1}}.
\end{equation}
Then \eqref{derli} implies
\begin{equation}\label{redl}
    A_{m+1}(z)=(1-z)^{m+2} \frac{d}{dz}\left( \frac{z \cdot A_m(z)}{(1-z)^{m+1}} \right).
\end{equation}
We may therefore
 recursively define the functions $A_m$ by
\begin{equation} \label{redl2}
    A_0(z):=1, \qquad A_{m+1}(z):=\bigl(1+ m z\bigr) A_{m}(z)+z(1-z)A'_{m}(z) \qquad (m \in \Z_{\gqs 0}).
\end{equation}
Hence  $A_m(z)$ is a polynomial and for $m \in \Z_{\gqs 1}$ it has degree $m-1$. These are the {\em Eulerian polynomials},
 introduced by Euler in connection with evaluating the Riemann zeta function at negative integers.
For example
\begin{equation*}
    A_1(z)=1, \quad A_2(z)=1+z,  \quad A_3(z)=1+4z+z^2, \quad A_4(z)=1+11z+11z^2+z^3.
\end{equation*}
The coefficient of $z^k$ in $A_m(z)$ has a combinatorial interpretation as the number of permutations of $\{1,2,\dots,m\}$ with $k$ ascents, see \cite[Sect. 6.2]{Knu}. Frobenius also showed in \cite{Fro} that
\begin{equation*}
    A_m(z) = \sum_{k=0}^m k! \stirb{m}{k} (z-1)^{m-k} \qquad \quad (m \in \Z_{\gqs 0})
\end{equation*}
with the Stirling number $\stirb{m}{k}$   indicating the number of partitions of  $m$ elements into $k$ non-empty subsets.

We have
\begin{align}\label{recip}
    z^{m-1}A_m(1/z) & = A_m(z) \qquad \quad (m \in \Z_{\gqs 1})
\intertext{which   is equivalent to}
    (-1)^{m+1}\pl_{-m}(1/z) & = \pl_{-m}(z) \qquad \quad (m \in \Z_{\gqs 1})  \label{rec9}
\end{align}
and \eqref{rec9} may be easily established  using \eqref{derli} and induction.

Thus by \eqref{fro},  $\pl_{-m}(z)$ is a rational function and its zeros are at $z=0$ along with the zeros of the Eulerian polynomial $A_m(z)$.  Frobenius showed in \cite{Fro} that the zeros of $A_m(z)$ are always distinct and negative. For completeness, we give a short proof that also shows their interlacing property. This proof is based on \cite[Thm. 4]{Pey} in a special case -- see also \cite[Thm. 1.34]{Bona} for a slightly weaker result.

\begin{theorem} \label{eu3}
For $m \in \Z_{\gqs 1}$ the Eulerian polynomial $A_m(z)$ has $m-1$ distinct zeros.  If $A_m(\rho)=0$ then $\rho < 0$ and $A_m(1/\rho)=0$. Also, exactly one zero of $A_m(z)$ lies between each pair of consecutive zeros of $A_{m+1}(z)$.
\end{theorem}
\begin{proof}
Suppose, for our induction hypothesis that $A_m(z)$ has $m-1$ distinct negative zeros. Note that by \eqref{redl2}, $z=1$ is not a zero of $A_{m+1}(z)$. Therefore, by \eqref{redl}, the zeros of $A_{m+1}(z)$ must be the zeros of the derivative of
 \begin{equation}\label{roll}
    z \cdot A_m(z) \cdot (1-z)^{-m-1}.
 \end{equation}
By Rolle's Theorem, this derivative has zeros between the zeros of $A_m(z)$ as well as one between the greatest zero of $A_m(z)$ and $0$. Since \eqref{roll}  goes to $0$ as $z \to -\infty$, the derivative is also zero at a value less than the least zero of $A_m(z)$. This accounts for all $m$ zeros of $A_{m+1}(z)$ and completes the induction. That the zeros come in reciprocal pairs follows from \eqref{recip}.
\end{proof}

In \cite[Thm. 4]{Pey}, the above theorem is extended to all negative real numbers:

\begin{theorem}[Peyerimhoff, 1966] \label{peyer}
For $r< 0$, $\pl_r(z)$ has $-\lfloor r \rfloor$ simple zeros for $z \in \C - [1,\infty)$ and they are all $\lqs 0$.
\end{theorem}

Sobolev, in \cite{Sob77}, seems to have been the first to locate of the zeros of $A_m(z)$ for $m$ large, with an explicit form for the error  given  in \cite{Sir78}. We give this result  next along with  the proof, since  \cite{Sob77,Sir78} give only brief summaries. Further results on these zeros appear in \cite{Sob78,Sob79,Sobb79}.

%For fixed $m$, label the zeros of $A_m(z)$ as $\lambda_1 <\lambda_2 < \cdots$. We focus on the zeros less than $-1$ since the remaining zeros are their reciprocals (as well as $-1$ if $m$ is even).

\begin{theorem} \label{sobol}
Fix $M>1$.
Suppose $m$ is large enough that
\begin{equation*}
    K:=\left(1 + \frac{1}{4}\sqrt{9\pi^2+\log^2 M} \right)\left(\frac{\pi^2+\log^2 M}{9\pi^2+\log^2 M}\right)^{(m+1)/2} \lqs 1/3.
\end{equation*}
Label the zeros of $A_m(z)$ as $\lambda_{m-1} <\lambda_{m-2} < \cdots <\lambda_{1}<0$.
Then for each $j$  with $-M \lqs \lambda_j \lqs -1/M$ we have
\begin{equation*}
    \lambda_j = -\exp\left(-\pi \cot\left(\frac{\pi(2j+1+\varepsilon_j)}{2(m+1)}\right)\right) \qquad \text{for} \qquad |\varepsilon_j| \lqs 2K/3.
\end{equation*}
\end{theorem}
\begin{proof}
Lipschitz's formula (an application of Poisson summation) gives
\begin{equation}\label{lipz}
    \frac{\pi^{1-s}}{\G(1-s)}\sum_{n=1}^\infty \frac{e^{\pi n z}}{n^s} = \sum_{k \in \Z} (2ki-z)^{s-1} \qquad (\Re(z), \ \Re(s)<0)
\end{equation}
in a special case, see \cite[Sect. 37]{Ra}. With $z=i+t$ and $s=-m$  in \eqref{lipz} we obtain
\begin{equation}\label{lipz2}
    \frac{\pi^{1+m}}{m!}\pl_{-m}(-e^{\pi t}) = 2\Re \sum_{k =1}^\infty \frac{1}{((2k-1)i - t)^{m+1}} \qquad (m \in \Z_{\gqs 1}).
\end{equation}
Equation \eqref{lipz2} is now valid for all $t \in \R$, because of the convergence of the right side of \eqref{lipz2}, and clearly both sides of \eqref{lipz2} are smooth functions of $t$.
For $m$ large, the term with $k=1$ on right side of \eqref{lipz2} is largest. Let $R_m(t)$ be the rest of the series. Then
\begin{equation} \label{ze}
    \pl_{-m}(-e^{\pi t}) = 0 \iff 2\Re\frac{1}{(i -t)^{m+1}} + R_m(t) = 0
\end{equation}
where
\begin{equation*}
    |R_m(t)|  \lqs  \sum_{k=2}^\infty \frac{2}{((2k-1)^2+t^2)^{(m+1)/2}}
     \lqs   \frac{2}{(9+t^2)^{(m+1)/2}} +  \int_{2}^\infty \frac{2 dx}{((2x-1)^2+t^2)^{(m+1)/2}}
\end{equation*}
This last integral is
\begin{align*}
    \int_{3}^\infty \frac{du}{(u^2+t^2)^{(m+1)/2}} & \lqs \frac{1}{(9+t^2)^{(m-1)/2}}\int_{3}^\infty \frac{du}{u^2+t^2}\\
     & \lqs \frac{1}{(9+t^2)^{(m-1)/2}}\int_{0}^\infty \frac{du}{u^2+9+t^2} = \frac{\pi}{2(9+t^2)^{m/2}}
\end{align*}
and we have shown
\begin{equation} \label{cour}
    |R_m(t)| \lqs \frac{1}{(9+t^2)^{(m+1)/2}}\left(2 + \frac{\pi}{2}\sqrt{9+t^2} \right).
\end{equation}
Next write
\begin{equation*}
    i t+1 = \sqrt{1+t^2} \cdot e^{i\pi\theta} \quad \text{with} \quad t=\tan(\pi\theta) \quad \text{for} \quad  -\frac{\pi}2 < \theta < \frac{\pi}2
\end{equation*}
to get
\begin{equation*}
    2\Re\frac{1}{(i -t)^{m+1}} = \frac{2\cos\bigl(\pi(m+1)(\theta+1/2)\bigr)}{ (1+t^2)^{(m+1)/2}}.
\end{equation*}
Let
\begin{equation*}
    R_m^*(t):=\frac 12 (1+t^2)^{(m+1)/2} R_m(t)
\end{equation*}
and \eqref{ze} implies that $\pl_{-m}(-e^{\pi \tan(\pi \theta)}) = 0$ if and only if
\begin{equation} \label{ze2}
     \cos\bigl(\pi(m+1)(\theta+1/2)\bigr) + R_m^*(\tan(\pi \theta)) = 0.
\end{equation}
The values of $t$ we are interested in have $-M \lqs -e^{\pi t} \lqs -1/M$, which is equivalent to $\pi|t| \lqs \log M$, and we see from \eqref{cour} that $|R_m^*(t)| \lqs K \lqs 1/3$ for these values of $t$. The corresponding range of $\theta$ is $-X \lqs \theta \lqs X$ for
\begin{equation*}
    X:=\frac 1\pi \arctan\left( \frac{\log M}{\pi}\right), \quad 0< X <\frac \pi 2.
\end{equation*}

So now we study the left side of \eqref{ze2} for $-X \lqs \theta \lqs X$. Write $\theta$ uniquely as
\begin{equation} \label{the}
    \theta = \frac{2j-m+\varepsilon}{2(m+1)} \qquad \text{for} \qquad j \in \Z, \ -1<\varepsilon \lqs 1
\end{equation}
and all the zeros of $\cos\bigl(\pi(m+1)(\theta+1/2)\bigr)$ occur for $j \in \Z$ and $\varepsilon = 0$
since
\begin{equation} \label{the2}
    \cos\bigl(\pi(m+1)(\theta+1/2)\bigr) = (-1)^{j+1}\sin(\pi \varepsilon/2).
\end{equation}

\begin{lemma} \label{also}
If
$\pl_{-m}(-e^{\pi \tan(\pi \theta)}) = 0$ for $-X \lqs \theta \lqs X$ then $\theta$ is of the form \eqref{the}
for some $j \in \Z$ and $\varepsilon$ satisfying $|\varepsilon| \lqs 2K/3$.
\end{lemma}
\begin{proof}
With \eqref{the2} and \eqref{ze2}
\begin{equation*}
    \bigl|\sin(\pi \varepsilon/2)\bigr| = \Bigl|\cos\bigl(\pi(m+1)(\theta+1/2)\bigr)\Bigr| = \bigl|R_m^*(\tan(\pi \theta))\bigr| \lqs K.
\end{equation*}
Recalling that $\bigl|\sin(\pi x)\bigr| \gqs 3|x|$ for $|x|\lqs 1/6$ gives the desired inequality for $\varepsilon$.
\end{proof}

Set $\theta_j:= \frac{2j-m}{2(m+1)}$ for $j\in \Z$. We see now that, for $-X \lqs \theta \lqs X$, the left side of \eqref{ze2} is possibly zero only for $\theta$ in intervals of the form $\bigl[\theta_j-\frac{K}{3(m+1)}, \theta_j+\frac{K}{3(m+1)}\bigr]$. It is also clear from Lemma \ref{also} and \eqref{the2} that outside these intervals the left side of \eqref{ze2} alternates $>0$ and $<0$. Therefore there is at least one zero in each such interval.

The next lemma shows there is at most one zero for $\theta$ in each of these intervals -- this point was not addressed in \cite{Sob77,Sir78}.

\begin{lemma}
The left side of \eqref{ze2} is strictly increasing or decreasing for $-X \lqs \theta \lqs X$ when $\theta$ is in the interval $\bigl[\theta_j-\frac{K}{3(m+1)}, \theta_j+\frac{K}{3(m+1)}\bigr]$.
\end{lemma}
\begin{proof}
A short computation, using $\cos^2(\pi \theta) = 1+t^2$, shows that $\frac{d}{d\theta}$ of the left of \eqref{ze2} may be expressed as
\begin{equation} \label{ze3}
     \pi(m+1)\left[-\sin\bigl(\pi(m+1)(\theta+1/2)\bigr) + \frac{t}{(1+t^2)^2}R_m^*(t)  + \frac{1}{(1+t^2)^2}R_{m+1}^*(t)\right].
\end{equation}
Write $\theta$ in the form \eqref{the} with $|\varepsilon| \lqs 2K/3$. Since
\begin{gather*}
    \Bigl|\sin\bigl(\pi(m+1)(\theta+1/2)\bigr)\Bigr|  = \bigl|\cos(\pi \varepsilon/2)\bigr|,\\
    \left|\frac{t}{(1+t^2)^2}R_m^*(t)  + \frac{1}{(1+t^2)^2}R_{m+1}^*(t)\right| <2K
\end{gather*}
we see that  \eqref{ze3} is non-zero if $|\cos(\pi \varepsilon/2)|>2K$ and this is equivalent to $1-\sin^2(\pi \varepsilon/2) > 4K^2$.
With $|\sin(x)|\lqs |x|$ we see that $|\sin(\pi \varepsilon/2)| \lqs \pi K/3$ for $|\varepsilon| \lqs 2K/3$. Hence
\begin{equation} \label{kp}
    K^2 < 1/(4+\pi^2/3^2)
\end{equation}
implies \eqref{ze3} is non-zero. Inequality \eqref{kp} is true by our assumption $K \lqs 1/3$.
\end{proof}

We have therefore shown that every zero $\lambda$ of $\pl_{-m}(z)$, for $-M \lqs \lambda \lqs -1/M$, is of the form
\begin{align}
    \lambda  & = -\exp\left(\pi \tan\left(\pi \frac{2j-m+\varepsilon}{2(m+1)}\right)\right) \notag\\
    & = -\exp\left(-\pi \cot\left(\pi \frac{2j+1+\varepsilon}{2(m+1)}\right)\right) \qquad \text{with} \qquad |\varepsilon| \lqs 2K/3 \label{zrro}
\end{align}
and also that there is exactly one zero of the form \eqref{zrro}
when the interval $\bigl[\theta_j-\frac{K}{3(m+1)}, \theta_j+\frac{K}{3(m+1)}\bigr]$  is contained in $(-X,X)$.
It only remains to match \eqref{zrro} with the ordering of the zeros $\lambda_{m-1} < \cdots <\lambda_1$. (Recall from Theorem \ref{eu3} that $\lambda_{m-j}=1/\lambda_j$.) If $m$ is even then $j=m/2$ in \eqref{zrro} gives the zero of
$\pl_{-m}(z)$ closest to $z=-1$. In fact $z=\lambda_{m/2}=-1$ is the middle zero in this case. Also $j=m/2-1$ in \eqref{zrro} gives the next zero to the right and $j=m/2+1$ the next to the left. Hence $\lambda_j$ is given by \eqref{zrro}. Similarly for $m$ odd. This completes the proof of Theorem \ref{sobol}.
\end{proof}

The above theorem is also essentially equivalent to \cite[Thm. 4]{Gaw} and we have used their ordering of the zeros since it generalizes to other $s$ values more readily. As an example, take $M=1000$ and $m=10$, giving $K \approx 0.034$. The leftmost two zeros of $A_{10}(z)$ are $\lambda_9\approx -963.85$, $\lambda_8\approx -37.54$ and all zeros are between $-M$ and $-1/M$. By comparison, the values of $-\exp\left(-\pi \cot\left(\pi \frac{2j+1}{2(m+1)}\right)\right)$ for $j=9,8$ are $-971.78, -37.55$. We have $|\varepsilon_j| < 0.0032$ for $1\lqs j\lqs 9$ and this  is less than $2K/3 \approx 0.023$.

%*************************************
% Graphics polylog zeros

\SpecialCoor
\psset{griddots=5,subgriddiv=0,gridlabels=0pt}
\psset{xunit=6cm, yunit=6cm}
\psset{linewidth=1pt}
\psset{dotsize=3pt 0,dotstyle=*}

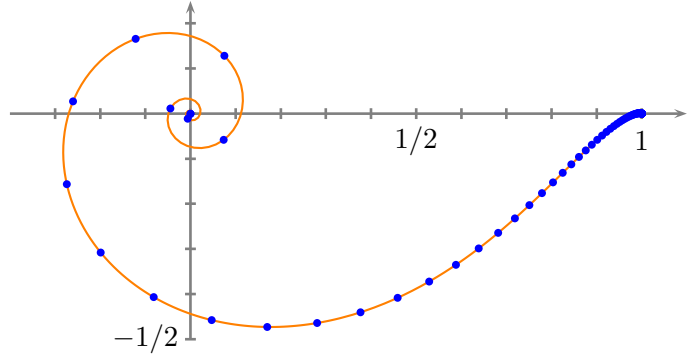
\begin{figure}[h]
\begin{center}
\begin{pspicture}(-0.3,-0.5)(1.1,0.2) %\psgrid

\savedata{\mydata}[
{{0,0},{0.000265119, -0.000616217}, {-0.00611432, -0.0112751}, {-0.0444369,
0.0113761}, {0.0736261, -0.0581044}, {0.075204, 0.128289},
{-0.121428, 0.165837}, {-0.260154, 0.0271782}, {-0.273929,
-0.156466}, {-0.198853, -0.308035}, {-0.0814242, -0.40684},
{0.0470798, -0.457835}, {0.170144, -0.473086}, {0.28071, -0.464271},
{0.376847, -0.440542}, {0.459084, -0.408438}, {0.528953, -0.372386},
{0.588251, -0.335271}, {0.638704, -0.298904}, {0.68183, -0.264369},
{0.718906, -0.232266}, {0.750982, -0.20288}, {0.778905, -0.176291},
{0.803356, -0.152453}, {0.824881, -0.131239}, {0.843917, -0.112481},
{0.860818, -0.0959848}, {0.875871, -0.08155}, {0.889309, -0.0689764},
{0.901327, -0.0580709}, {0.912089, -0.0486512}, {0.921731,
-0.0405473}, {0.930374, -0.033603}, {0.938119, -0.0276759},
{0.945058, -0.0226372}, {0.95127, -0.0183715}, {0.956827,
-0.0147755}, {0.961795, -0.0117577}, {0.966229, -0.00923717},
{0.970184, -0.0071427}, {0.973707, -0.00541199}, {0.976841,
-0.00399065}, {0.979626, -0.00283141}, {0.982097, -0.00189337},
{0.984287, -0.00114122}, {0.986225, -0.000544617}, {0.987939,
-0.0000775486}, {0.989452, 0.000282181}, {0.990786, 0.000553458},
{0.991961, 0.000752293}, {0.992994, 0.000892211}, {0.993902,
0.000984589}, {0.994699, 0.00103896}, {0.995397, 0.00106329},
{0.996008, 0.00106418}, {0.996542, 0.00104709}, {0.997009,
0.00101651}, {0.997416, 0.000976121}, {0.99777, 0.000928886},
{0.998079, 0.000877201}, {0.998347, 0.000822976}, {0.998579,
0.000767713}, {0.998781, 0.000712583}, {0.998956, 0.000658477},
{0.999107, 0.000606061}, {0.999237, 0.000555816}, {0.99935,
0.000508075}, {0.999447, 0.000463049}, {0.99953, 0.000420854},
{0.999602, 0.000381532}, {0.999663, 0.000345067}, {0.999716,
0.000311398}, {0.999761, 0.000280432}, {0.999799, 0.000252052},
{0.999832, 0.000226125}, {0.99986, 0.00020251}, {0.999884,
0.000181058}, {0.999904, 0.000161619}, {0.999921, 0.000144047},
{0.999935, 0.000128195}, {0.999947, 0.000113926}, {0.999957,
0.000101106}, {0.999965, 0.0000896086}, {0.999972, 0.0000793154},
{0.999978, 0.0000701156}, {0.999983, 0.0000619061}, {0.999987,
0.0000545913}, {0.99999, 0.0000480834}, {0.999993, 0.0000423015},
{0.999995, 0.0000371716}, {0.999997, 0.0000326263}, {0.999998,
0.0000286041}, {0.999999, 0.0000250493}, {1., 0.0000219115}, {1.,
0.0000191452}, {1., 0.0000167092}, {1., 0.0000145666}, {1.,
0.0000126842}}
]

\savedata{\mydatb}[
{{0.0000334438,
  0.0000319987}, {-0.0000855334, -0.0000368327}, {0.000162703,
-0.0000528502}, {-0.0000690329,
  0.00028347}, {-0.000399868, -0.000243028}, {0.000430212,
-0.000568456}, {0.000872203, 0.00056641}, {-0.000509639,
  0.00137031}, {-0.00199091, -0.0000178812}, {-0.00114306,
-0.00237721}, {0.00183954, -0.00287335}, {0.00431111, -0.000292704},
{0.00377505, 0.00382215}, {-0.0000951439, 0.00656939}, {-0.00539562,
  0.00579567}, {-0.00932978,
  0.0012981}, {-0.00973117, -0.00528692}, {-0.00594346, -0.0114305},
{0.00112518, -0.014802}, {0.00952593, -0.0140284}, {0.0169993,
-0.00896271}, {0.0216159, -0.000539517}, {0.0221763,
  0.00960996}, {0.0183469, 0.019607}, {0.0105841,
  0.0277156}, {-0.0000690327, 0.0326107}, {-0.0122359,
  0.0335124}, {-0.0244494, 0.0302036}, {-0.0353491,
  0.0229651}, {-0.0438179,
  0.0124624}, {-0.0490574, -0.000386803}, {-0.0506114, -0.0145399},
{-0.0483494, -0.0289409}, {-0.0424244, -0.0426067}, {-0.0332149,
-0.0546878}, {-0.0212635, -0.0645043}, {-0.00721569, -0.0715617},
{0.00823396, -0.0755502}, {0.0243876, -0.076332}, {0.0405797,
-0.0739214}, {0.0562038, -0.0684598}, {0.0707298, -0.0601904},
{0.0837142, -0.0494318}, {0.0948041, -0.036554}, {0.103736,
-0.0219579}, {0.110332, -0.00605589}, {0.114492,
  0.0107425}, {0.116184, 0.028043}, {0.11544, 0.0454754}, {0.112338,
  0.0627005}, {0.107001, 0.0794138}, {0.099582,
  0.0953486}, {0.0902588, 0.110276}, {0.0792261,
  0.124006}, {0.0666892, 0.136384}, {0.052858, 0.147291}, {0.0379427,
  0.15664}, {0.02215, 0.164372}, {0.00567953, 0.170455}, {-0.0112778,
  0.174883}, {-0.028542, 0.177666}, {-0.0459457,
  0.178834}, {-0.063334, 0.178432}, {-0.080566,
  0.176517}, {-0.0975142, 0.173155}, {-0.114065,
  0.168421}, {-0.130117, 0.162395}, {-0.145583, 0.155163}, {-0.160387,
   0.146811}, {-0.174465, 0.137428}, {-0.187763, 0.127104}, {-0.20024,
   0.115929}, {-0.21186, 0.103988}, {-0.222598,
  0.0913669}, {-0.232436, 0.078149}, {-0.241364,
  0.0644138}, {-0.249378, 0.0502374}, {-0.256478,
  0.0356926}, {-0.262671, 0.0208483}, {-0.267967,
  0.00576958}, {-0.272381, -0.00948244}, {-0.27593, -0.0248506},
{-0.278635, -0.0402816}, {-0.280517, -0.0557259}, {-0.281603,
-0.0711381}, {-0.281918, -0.0864759}, {-0.28149, -0.101701},
{-0.280346, -0.116778}, {-0.278518, -0.131676}, {-0.276033,
-0.146365}, {-0.272922, -0.16082}, {-0.269215, -0.175017},
{-0.264941, -0.188936}, {-0.26013, -0.20256}, {-0.254812, -0.215872},
{-0.249015, -0.228859}, {-0.242767, -0.241509}, {-0.236096,
-0.253813}, {-0.229029, -0.265763}, {-0.221591, -0.277352},
{-0.213808, -0.288575}, {-0.205705, -0.29943}, {-0.197305,
-0.309913}, {-0.188631, -0.320024}, {-0.179704, -0.329762},
{-0.170547, -0.339129}, {-0.161179, -0.348125}, {-0.151619,
-0.356755}, {-0.141887, -0.365019}, {-0.131999, -0.372924},
{-0.121973, -0.380472}, {-0.111824, -0.387668}, {-0.101567,
-0.394519}, {-0.0912183, -0.401029}, {-0.0807901, -0.407205},
{-0.0702959, -0.413052}, {-0.0597478, -0.418578}, {-0.0491576,
-0.423788}, {-0.0385361, -0.428691}, {-0.0278937, -0.433293},
{-0.0172402, -0.4376}, {-0.00658456, -0.441621}, {0.00406449,
-0.445363}, {0.014699, -0.448832}, {0.0253113, -0.452036},
{0.0358945, -0.454983}, {0.0464419, -0.45768}, {0.0569474,
-0.460133}, {0.0674054, -0.46235}, {0.0778105, -0.464339},
{0.0881578, -0.466105}, {0.0984428, -0.467657}, {0.108661, -0.469},
{0.118809, -0.470143}, {0.128884, -0.47109}, {0.138881, -0.471849},
{0.148798, -0.472427}, {0.158632, -0.472829}, {0.168381, -0.473062},
{0.178042, -0.473131}, {0.187614, -0.473043}, {0.197095, -0.472803},
{0.206483, -0.472418}, {0.215777, -0.471891}, {0.224976, -0.47123},
{0.234078, -0.470438}, {0.243083, -0.469522}, {0.251991, -0.468486},
{0.2608, -0.467334}, {0.269509, -0.466073}, {0.27812, -0.464705},
{0.286631, -0.463237}, {0.295043, -0.461671}, {0.303355, -0.460013},
{0.311567, -0.458265}, {0.319681, -0.456433}, {0.327695, -0.454521},
{0.332983, -0.453202}, {0.358769, -0.446133}, {0.383481, -0.438359},
{0.407144, -0.429993}, {0.429788, -0.421138}, {0.451446, -0.411885},
{0.472154, -0.402316}, {0.491949, -0.392501}, {0.510869, -0.382506},
{0.528953, -0.372386}, {0.546238, -0.362192}, {0.562761, -0.351968},
{0.578559, -0.341753}, {0.593668, -0.331578}, {0.608121, -0.321475},
{0.621952, -0.311468}, {0.635192, -0.301578}, {0.647871, -0.291826},
{0.660018, -0.282226}, {0.671661, -0.272792}, {0.682826, -0.263535},
{0.693536, -0.254465}, {0.703816, -0.245589}, {0.713687, -0.236912},
{0.72317, -0.228441}, {0.732285, -0.220177}, {0.74105, -0.212124},
{0.749483, -0.204283}, {0.7576, -0.196653}, {0.765416, -0.189236},
{0.772947, -0.18203}, {0.780205, -0.175034}, {0.787203, -0.168247},
{0.793954, -0.161665}, {0.800469, -0.155286}, {0.806759, -0.149107},
{0.812834, -0.143125}, {0.818702, -0.137337}, {0.824374, -0.131739},
{0.829858, -0.126327}, {0.835161, -0.121097}, {0.840291, -0.116045},
{0.845255, -0.111168}, {0.850059, -0.10646}, {0.854711, -0.101918},
{0.859216, -0.0975383}, {0.863579, -0.0933156}, {0.867806,
-0.0892461}, {0.871901, -0.0853256}, {0.875871, -0.08155}, {0.879718,
-0.077915}, {0.883448, -0.0744167}, {0.887065, -0.071051}, {0.890572,
-0.0678139}, {0.893973, -0.0647014}, {0.897272, -0.0617097},
{0.900471, -0.0588351}, {0.903575, -0.0560738}, {0.906587,
-0.053422}, {0.909508, -0.0508764}, {0.912343, -0.0484333},
{0.915093, -0.0460893}, {0.917762, -0.043841}, {0.920351,
-0.0416853}, {0.922864, -0.0396189}, {0.925302, -0.0376387},
{0.927668, -0.0357417}, {0.929964, -0.0339249}, {0.932191,
-0.0321855}, {0.934353, -0.0305206}, {0.93645, -0.0289276},
{0.938484, -0.0274038}, {0.940458, -0.0259467}, {0.942374,
-0.0245537}, {0.944232, -0.0232225}, {0.946034, -0.0219507},
{0.947782, -0.0207361}, {0.949478, -0.0195764}, {0.951123,
-0.0184696}, {0.952718, -0.0174135}, {0.954265, -0.0164062},
{0.955765, -0.0154457}, {0.95722, -0.0145302}, {0.95863, -0.0136579},
{0.959997, -0.012827}, {0.961323, -0.0120359}, {0.962607,
-0.0112829}, {0.963853, -0.0105664}, {0.96506, -0.00988499},
{0.966229, -0.00923717}, {0.967363, -0.00862151}, {0.968461,
-0.00803666}, {0.969525, -0.00748132}, {0.970555, -0.00695422},
{0.971553, -0.00645415}, {0.97252, -0.00597992}, {0.973457,
-0.00553043}, {0.974364, -0.00510458}, {0.975242, -0.00470133},
{0.976092, -0.00431968}, {0.976915, -0.00395865}, {0.977711,
-0.00361733}, {0.978483, -0.00329482}, {0.979229, -0.00299026},
{0.979951, -0.00270284}, {0.98065, -0.00243175}, {0.981326,
-0.00217624}, {0.98198, -0.00193559}, {0.982613, -0.00170909},
{0.983225, -0.00149607}, {0.983817, -0.00129589}, {0.98439,
-0.00110793}, {0.984943, -0.000931609}, {0.985478, -0.00076635},
{0.985996, -0.000611614}, {0.986496, -0.000466882}, {0.98698,
-0.000331655}, {0.987447, -0.000205457}, {0.987899, -0.0000878319},
{0.988335, 0.0000216559}, {0.988757, 0.000123424}, {0.989164,
0.000217871}, {0.989558, 0.00030538}, {0.989938, 0.000386313},
{0.990305, 0.00046102}, {0.99066, 0.000529834}, {0.991002,
0.000593071}, {0.991333, 0.000651036}, {0.991652, 0.000704018},
{0.991961, 0.000752293}, {0.992258, 0.000796126}, {0.992546,
0.000835766}, {0.992823, 0.000871453}, {0.99309, 0.000903417},
{0.993348, 0.000931874}, {0.993598, 0.000957032}, {0.993838,
0.000979087}, {0.99407, 0.000998228}, {0.994293, 0.00101463}}
]

\psline[linecolor=gray]{->}(-0.4,0)(1.1,0)
\psline[linecolor=gray]{->}(0,-0.5)(0,0.25)
\multirput(-0.3,-0.012)(0.1,0){14}{\psline[linecolor=gray](0,0)(0,0.024)}
\multirput(-0.012,-0.5)(0,0.1){8}{\psline[linecolor=gray](0,0)(0.024,0)}

\dataplot[linecolor=orange,linewidth=0.8pt,plotstyle=line]{\mydatb}
\dataplot[linecolor=blue,linewidth=0.8pt,plotstyle=dots]{\mydata}

\rput(-0.1,-0.5){$-1/2$}
\rput(0.5,-0.06){$1/2$}
\rput(1,-0.06){$1$}

\end{pspicture}
\caption{The zeros of $\pl_{s}(z)$ for $s=-10-44i$}\label{zfig}
\end{center}
\end{figure}
%*************************************

 Gawronski and Stadtm\"uller in \cite{Gaw} generalize Theorems \ref{sobol} and \ref{peyer} to the zeros of $\pl_s(z)$ for  $s \in \C$ with $\Re(s)<0$. The required computations become much more elaborate in this  general case. Though the work in \cite{Gaw} is independent of \cite{Sob77}, it is based on the same essential idea.
Starting with \eqref{lipz}, replace $z$ with $z+i$ and ignore the terms with $k \neq 0,1$ to show
\begin{equation*}
    \frac{\pi^{1-s}}{\G(1-s)}\pl_s(-e^{\pi z}) \approx \frac{1}{(i-z)^{1-s}}+\frac{1}{(-i-z)^{1-s}}
\end{equation*}
where the error may be explicitly bounded. Therefore
\begin{align*}
    \pl_s(-e^{\pi z}) = 0 & \iff 1+\left(\frac{i-z}{-i-z}\right)^{1-s} \approx 0 \\
    & \iff \frac{i-z}{-i-z} \approx \exp\left( \frac{\pi i(2j+1)}{1-s}\right) \qquad (j\in \Z)\\
    & \iff z \approx -\cot\left( \frac{\pi (2j+1)}{2(1-s)}\right) \qquad (j\in \Z)
\end{align*}
and we expect the zeros of $\pl_s(z)$ to approximately take the form
\begin{equation}\label{plzo}
    -\exp\left(-\pi \cot\left( \frac{\pi (2j+1)}{2(1-s)}\right)\right).
\end{equation}
Due to branch considerations, \eqref{plzo} will correspond to zeros of $\pl_s(z)$ only for certain integers $j \gqs 0$, see \cite[Lemma 2, Thm. 2]{Gaw}.

Figure \ref{zfig} shows the zeros of $\pl_{s}(z)$ for $s=-10-44i$, found numerically. They are indistinguishable, at the scale of the figure, from the points $0$ and \eqref{plzo} for $j=0,1, \cdots, 139$.  This $s$ satisfies $s=(1+v)(1+4i)-1$ for $v=10$. If we change $s$ by increasing $v$, then the number of zeros of $\pl_{s}(z)$ increases, all getting closer to the spiral shown in Figure \ref{zfig} and filling it more densely. See \cite[Thm. 2]{Gaw} for  precise statements.

It should be possible to extend the results in \cite{Gaw} to all branches. We have seen that $\li(z)$ has infinitely many zeros if we include those of every branch, while $\pl_1(z)$ has only one zero. A natural question arises: for any fixed $s \in \C$, how many zeros does $\pl_s(z)$ have on all branches?

{\small
\bibliography{dilogzerosdataupdated}
}

\textsc{Dept. of Mathematics, The CUNY Graduate Center , New York, NY 10016-4309, U.S.A.}

{\em E-mail address:} \texttt{cosullivan@gc.cuny.edu}

{\em Web page:} \texttt{http://fsw01.bcc.cuny.edu/cormac.osullivan}

\end{document}